\documentclass{amsart}

\usepackage{amsmath}
\usepackage{amsthm}
\usepackage{amsfonts}
\usepackage{dsfont}
\usepackage{color}
\usepackage{enumerate}
\usepackage{amssymb}
\usepackage{booktabs}
\usepackage{arydshln}
\usepackage{blkarray}
\usepackage{hhline}
\usepackage[T1]{fontenc}
\usepackage{mathtools}
\usepackage[margin=36mm]{geometry}
\usepackage{enumerate}
\usepackage{tasks}
\usepackage[utf8]{inputenc}
\usepackage[english]{babel}
\newcommand{\listintertext}{\@ifstar\listintertext@\listintertext@@}

\newtheorem{theorem}{Theorem}[section]

\newtheorem{lemma}[theorem]{Lemma}
\newtheorem{proposition}[theorem]{Proposition}
\newtheorem{example}[theorem]{Example}
\theoremstyle{definition}
\newtheorem{definition}[theorem]{Definition}
\theoremstyle{remark}
\newtheorem{remark}[theorem]{Remark}
\numberwithin{equation}{section}

\title[]{Morphisms, direct sums and tensor products\\ of K\"ahler--Poisson algebras}

\author{Ahmed Al-Shujary}
\address{Dept. of Mathematics, Link\"oping Univeristy, 58183 Link\"oping, Sweden.}
\email{ahmed.al-shujary@liu.se}

\thanks{}

\subjclass[2000]{}

\allowdisplaybreaks

\begin{document}

\maketitle
\begin{abstract}
  In this paper we introduce the concept of morphisms of
  K\"ahler-Poisson algebras and study their algebraic properties. In
  particular, we find conditions, in terms of the metric, for two
  algebras to be isomorphic, and we introduce direct sums and tensor
  products of K\"ahler-Poisson algebras. We provide detailed examples
  to illustrate the novel concepts.
\end{abstract}

\section{Introduction}

The study of geometry via Poisson algebras goes back to two centuries
ago through the works of Lagrange, Poisson and Lie. Poisson
\cite{Poisson} invented his brackets as a tool for classical dynamics,
Jacobi \cite{Jacobi} realized the importance of these brackets and
studied their algebraic properties, and Lie \cite{Lie} began the study
of their geometry. The study of geometry via Poisson algebras has
experienced an amazing development since the $1980^,$s, starting with
the foundational work of Weinstein \cite{weinstein} on Poisson
manifolds. Since then many authors have studied the geometric and
algebraic properties of symplectic and Poisson manifolds (see
e.g. \cite{brylinski,h:poisson.cohomology,h:extensions.rinehart,deformationquantization,Lichnerowicz}).
Later, Kontsevich \cite{deformationquantization} has shown that the
classification of formal deformations of the algebra of functions for
any manifold $\Sigma$ is equivalent to the classification of formal
families of Poisson structures on $\Sigma$.

However, metric aspects of Poisson manifolds have not been
investigated to the same extent. In \cite{algebras} we began to study
these metric aspects by defining K\"ahler-Poisson algebras as
algebraic analogues of the algebra of functions on K\"ahler manifolds
(the concept of K\"ahler manifolds was introduced by E. K\"ahler
\cite{E.K}). This study of metric aspects was motivated by the results
in \cite{Pseudo-Riemannian,multi-linear}, where many aspects of the
differential geometry of embedded Riemannian manifolds $\Sigma$ can be
formulated in terms of the Poisson structure of the algebra of
functions of $\Sigma$. In \cite{algebras} we showed that ``the
K\"ahler--Poisson condition'', being the crucial identity in the
definition of K\"ahler-Poisson algebras, allowed for an identification of
geometric objects in the Poisson algebra which share crucial
properties with their classical counterparts. For instance, we proved
the existence of a unique Levi-Civita connection on the module
generated by the inner derivations of the K\"ahler-Poisson algebra,
and show that the curvature operator has all the classical symmetries.
 
It is generally interesting to ask how many different (up to
isomorphism) K\"ahler-Poisson structures do there exist on a given
Poisson algebra?  The aim of this paper is to explore further
algebraic properties of K\"ahler-Poisson algebras. In particular, we
find appropriate definitions of morphisms of K\"ahler-Poisson
algebras. We illustrate with examples when two K\"ahler-Poisson
algebras are isomorphic, for instance, we begin by taking algebras
$\mathcal{A}$ and $\mathcal{A}^{\prime}$, where $\mathcal{A}$ is
finitely generated algebra, $\mathcal{A}^{\prime}=\mathcal{A}$ and we
consider different set of generators for the K\"ahler-Poisson algebra
structures of $\mathcal{A}$ and $\mathcal{A}^{\prime}$.  We then use
the concept of morphism to define subalgebras of K\"ahler-Poisson
algebras.  Again, we present examples to understand subalgebras of
K\"ahler-Poisson algebras. Finally, we introduce direct sums and
tensor products of K\"ahler-Poisson algebras together with their basic
properties.

\section{K\"ahler-Poisson algebras}

We begin this section by recalling the main object of our
investigation. In \cite{algebras} we introduce K\"ahler-Poisson
algebras over the field $\mathbb{K}$ (either $\mathbb{R}$ or
$\mathbb{C}$). Let us consider a Poisson algebra
$(\mathcal{A},\{.,.\})$, over a field $\mathbb{K}$ and let
$\{x^1,...,x^m\}$ be a set of distinguished elements of
$\mathcal{A}$. These elements play the role of functions providing an
embedding into $\mathbb{R}^m$ for K\"ahler manifolds. K\"ahler
manifolds and their properties have been extensively studied (see
e.g. \cite{brylinski,Lichnerowicz,weinstein,wellsjr}). Let us recall
the definition of K\"ahler-Poisson algebras together with a few basic
results.
\begin{definition} 
\label{def1}
Let $(\mathcal{A},\{\cdot,\cdot\})$ be a Poisson algebra over
$\mathbb{K}$ and let ${x^1,...,x^m} \in \mathcal{A}$. Given a
symmetric $m \times m$ matrix $g=(g_{ij})$ with entries
$g_{ij}\in\mathcal{A}$, for $i,j=1,...,m$, we say that the triple
$\mathcal{K}=(\mathcal{A},g,\{x^1,...,x^m\})$ is a \emph{K\"ahler--Poisson
algebra} if there exists $\eta\in\mathcal{A}$ such
that \begin{equation} \label{eq:2.1} \sum\limits_{i,j,k,l}^m
  \eta\{a,x^i\}g_{ij}\{x^j,x^k\}g_{kl}\{x^l,b\}=
  -\{a,b\} \end{equation} for all $a,b$ $\in$ $\mathcal{A}$. We call
equation \eqref{eq:2.1} \emph{``the K\"ahler--Poisson condition''}.
\end{definition}

Given a K\"ahler-Poisson algebra
$\mathcal{K}=(\mathcal{A},g,\{x^1,...,x^m\})$, let $\mathfrak{g}$
denote the $\mathcal{A}$-module generated by all
$\emph{inner derivations}$ , i.e.
\begin{center}  
$\mathfrak{g}=\{a_1\{c^1,.\}+...+a_N\{c^N,.\}:$ $a_i,c^i \in \mathcal{A}$ and $N\in \mathbb{N}\}$. 
\end{center}
It is a standard fact that $\mathfrak{g}$ is a Lie algebra over
$\mathbb{K}$ with respect to the bracket
\begin{center}
$[\alpha,\beta](a)=\alpha(\beta(a))-\beta(\alpha(a))$,
\end{center}
where $\alpha,\beta \in \mathfrak{g}$ and $a \in \mathcal{A}$ (see e.g.\cite{helgason,Lee}).

It was shown in \cite{algebras} that $\mathfrak{g}$ is a projective
module and that every K\"ahler--Poisson algebra is a Lie-Rinehart
algebra. More details for Lie-Rinehart algebra can be found in
\cite{h:extensions.rinehart,rinehart}. Moreover, it follows that the matrix
$g$ defines a metric on $\mathfrak{g}$ by \begin{center}
  $g(\alpha,\beta)=\alpha(x^i)g_{ij}\beta(x^j)$,
\end{center}
for $\alpha,\beta \in \mathfrak{g}$.

Let us now introduce some notation for  K\"ahler--Poisson algebras. We set
 \begin{center}
 $\mathbb{\mathcal{P}}^{ij}$=$\{x^i,x^j\}$  
\end{center} 
and
\begin{center}
	$\mathbb{\mathcal{P}}^i(a)$=$\{x^i,a\}$ 
\end{center} for $a\in \mathcal{A}$, as well as

\begin{center}
	$\mathcal{D}^{ij}=\eta \{x^i,x^l\}g_{lk}\{x^j,x^k\}$ \\ $\mathcal{D}^{i}(a) =\eta \{x^k,a\}g_{kl}\{x^l,x^i\}$.
\end{center}
Note that $\mathcal{D}^{ij}=\mathcal{D}^{ji}$.  The metric $g$ will be
used to lower indices in analogy with differential
geometry. E.g.
\begin{center}
  $\mathcal{P}^i_{\;\;j}=\mathcal{P}^{ik}g_{kj} \quad
  \mathcal{D}^i_{\;\;j}=\mathcal{D}^{ik}g_{kj} \quad
  \mathcal{D}_i=g_{ij}\mathcal{D}^j$.
\end{center}
In this notation, \eqref{eq:2.1} can be stated as 
\begin{equation}
  \mathcal{D}_{i}(a)\mathcal{P}^{i}(b)=\{a,b\}.
\end{equation}
Furthermore, one immediately derives the following identities 
\begin{equation} 
\mathcal{D}^{ij}\mathcal{P}_j(a)=\mathcal{P}^i(a),\\\\\\\ \mathcal{P}^{ij}\mathcal{D}_j(a)=\mathcal{P}^i(a)\\\\\ \text{and}    \\\\\  \mathcal{D}^{i}_j\mathcal{D}^{jk}=\mathcal{D}^{jk}.
\end{equation} 
If we denote by $\mathcal{P}$ the matrix with entries
$\mathcal{P}^{ij}$, the K\"ahler-Poisson condition \eqref{eq:2.1} in
Definition \ref{def1}, can be written in matrix notation as
\begin{center}
$\eta\mathcal{P}g\mathcal{P}g\mathcal{P}=-\mathcal{P}$,
\end{center}
for an algebra $\mathcal{A}$ generated by $\{x^1,...,x^m\}$.

Let us also recall Example 6.1 in \cite{algebras} (see also
\cite{thesis}). This example shows that any Poisson algebra generated
by two elements can be endowed with a K\"ahler-Poisson algebra structure.
\begin{example}\cite{algebras}
\label{example.two elements}
Let $\mathcal{A}$ be a Poisson algebra generated by two elements $x$ and $y\in\mathcal{A}$ and let
\begin{center} 
$\mathcal{P}=\begin{pmatrix}
 0 & \{x,y\}  \\
  -\{x,y\} & 0 \\
\end{pmatrix}$.\end{center} It is easy to check that for an arbitrary symmetric matrix with ($\det(g) \neq 0$)
\begin{center}
 $g=\begin{pmatrix}
 a & c  \\
  c & b \\
\end{pmatrix}$
\end{center}
one obtains
\begin{align*}
  \mathcal{P}g\mathcal{P}g &=\begin{pmatrix}
    c\{x,y\} & b\{x,y\}  \\
    -a\{x,y\}  & -c\{x,y\} \\
  \end{pmatrix}  
  \begin{pmatrix}
    c\{x,y\} & b\{x,y\}  \\
    -a\{x,y\}  & -c\{x,y\} \\
  \end{pmatrix}  
  \\&=\begin{pmatrix}
    (c\{x,y\})^2-ab\{x,y\}^2& 0 \\
    0 & 	(c\{x,y\})^2-ab\{x,y\}^2\\
  \end{pmatrix}.
\end{align*}
Therefore
\begin{align*}
  \mathcal{P}g\mathcal{P}g\mathcal{P} &=\begin{pmatrix}
    0& \{x,y\}(	(c\{x,y\})^2-ab\{x,y\}^2)  \\
    -\{x,y\}(	(c\{x,y\})^2-ab\{x,y\}^2)  & 0\\
  \end{pmatrix} 
  \\&= -\{x,y\}^2(ab-c^2)\mathcal{P}
  \\&= -\{x,y\}^2 \det(g)\mathcal{P},
\end{align*} 
giving $\eta=(\{x,y\}^2\det(g))^{-1}$, provided that the inverse
exists. Thus, as long as $\{x,y\}^2 \det(g)$ is not a zero-divisor,
one may localize $\mathcal{A}$ to obtain a K\"ahler-Poisson algebra
\begin{center}
$\mathcal{K}=(\mathcal{A}[(\{x,y\}^2\det(g))^{-1}],g,\{x,y\})$.
\end{center}
\end{example}

\section{Homomorphisms of K\"ahler-Poisson algebras} 

In this section we are interested in studying maps between
K\"ahler-Poisson algebras. Given a Poisson algebra, we will
investigate isomorphism classes of K\"ahler-Poisson algebra structures
on the Poisson algebra. As the definition of a K\"ahler-Poisson
algebra involves the choice of a set of distinguished elements, we
will require a morphism to respect the subalgebra generated by these
elements. To this end, we start by making the following definition:

\begin{definition}
  Given a K\"ahler-Poisson algebra
  $\mathcal{K}=(\mathcal{A},g,\{x^1,...,x^m\})$ on a Poisson algebra
  $(\mathcal{A},\{\cdot,\cdot\})$, let
  $\mathcal{A}_\text{fin} \subseteq \mathcal{A}$ denote the Poisson
  subalgebra generated by $\{x^1,...,x^m\}$.
\end{definition}

\noindent
We now introduce morphisms of  K\"ahler-Poisson algebras in the following way:
 
\begin{definition}
\label{definition}
Let $\mathcal{K}=(\mathcal{A},g,\{x^1,...,x^m\})$ and
$\mathcal{K}^{\prime}=(\mathcal{A}^{\prime},g^{\prime},\{y^1,...,{y^m}^{\prime}\})$
be K\"ahler-Poisson algebras together with their modules of
derivations $\mathfrak{g}$ and $\mathfrak{g}^{\prime}$, respectively.
A (homo)morphism of K\"ahler-Poisson algebras is a pair of maps
$(\phi,\psi)$, with $\phi:\mathcal{A}\rightarrow \mathcal{A}^{\prime}$
a Poisson algebra homomorphism and
$\psi:\mathfrak{g}\rightarrow \mathfrak{g}^{\prime}$ a Lie algebra
homomorphism, such that

\begin{enumerate}
  \item \label{1}$\psi(a\alpha)=\phi(a)\psi(\alpha)$,
  \item \label{2}$\phi(\alpha(a))=\psi(\alpha)(\phi(a))$,
   \item \label{3}$\phi(g(\alpha,\beta))=g^{\prime}(\psi(\alpha),\psi(\beta))$,
   \item \label{4}$\phi(\mathcal{A}_\text{fin})\subseteq \mathcal{A}^{\prime}_{\text{fin}} $,
\end{enumerate}

for all $a \in \mathcal{A}$ and $\alpha,\beta \in \mathfrak{g}$.

\end{definition}

\begin{remark}
  Note that a morphism of K\"ahler-Poisson algebras is also a morphism
  of the underlying Lie-Rinehart algebras.
\end{remark}

Furthermore, in next Proposition we show that the composition of two
K\"ahler-Poisson algebras morphisms is a morphism of K\"ahler-Poisson
algebras as we expect and require.

\begin{proposition}
\label{pro.3.6}
Let $\mathcal{K}=(\mathcal{A},g,\{x^1,...,x^m\})$ and
$\mathcal{K}^{\prime}=(\mathcal{A}^{\prime},g^{\prime},\{y^1,...,{y^m}^{\prime}\})$
be K\"ahler-Poisson algebras together with their corresponding modules
of derivations $\mathfrak{g}$, $\mathfrak{g}^{\prime}$ and let
$\mathcal{K^{\prime\prime}}=(\mathcal{A^{\prime\prime}},g^{\prime\prime},\{z^1,...,z^m{^{\prime\prime}}\})$
be a K\"ahler-Poisson algebra with its module
of derivations $\mathfrak{g}^{\prime\prime}$. If
$(\phi,\psi):\mathcal{K}\rightarrow \mathcal{K}^{\prime}$ and
$(\phi^{\prime},\psi^{\prime}):\mathcal{K}^{\prime}\rightarrow
\mathcal{K}^{\prime\prime}$ are homomorphisms of K\"ahler-Poisson
algebras then
$(\phi^{\prime}\circ \phi,\psi^{\prime}\circ
\psi):\mathcal{K}\rightarrow \mathcal{K}^{\prime\prime}$ is a
homomorphism of K\"ahler-Poisson algebras.
\end{proposition}
\begin{proof}
  Let $\hat{\phi}=\phi^{\prime}\circ\phi$ and
  $\hat{\psi}=\psi^{\prime}\circ\psi$, where
  $\phi:\mathcal{A}\rightarrow \mathcal{A}^{\prime}$,
  $\phi^{\prime}:\mathcal{A}^{\prime} \rightarrow
  \mathcal{A}^{\prime\prime}$,
  $\psi:\mathfrak{g}\rightarrow \mathfrak{g}^{\prime}$ and
  $\psi^{\prime}:\mathfrak{g}^{\prime}\rightarrow
  \mathfrak{g}^{\prime\prime}$, giving
  $\hat{\phi}:\mathcal{A} \rightarrow \mathcal{A}^{\prime\prime}$ and
  $\hat{\psi}:\mathfrak{g}\rightarrow \mathfrak{g}^{\prime\prime}$. To
  prove that $(\hat{\phi},\hat{\psi})$ is a homomorphism of
  K\"ahler-Poisson algebras we show that $(\hat{\phi},\hat{\psi})$
  satisfies properties (1-4) in Definition \ref{definition}.

  (1) $\hat{\psi}(a\alpha)=\hat{\phi}(a)\hat{\psi}(\alpha)$, for all
  $\alpha \in \mathfrak{g}$ and $a \in \mathcal{A}$.
  \begin{align*} 
    \psi^{\prime}(\psi(a\alpha))-\phi^{\prime}(\phi(a))\psi^{\prime}(\psi(\alpha))&=\psi^{\prime}(\phi(a)\psi(\alpha))-\phi^{\prime}(\phi(a))\psi^{\prime}(\psi(\alpha))\\&=\phi^{\prime}(\phi(a))\psi^{\prime}(\psi(\alpha))-\phi^{\prime}(\phi(a))\psi^{\prime}(\psi(\alpha))=0.
  \end{align*}

  (2) $\hat{\phi}(\alpha(a))=\hat{\psi}(\alpha)(\hat{\phi}(a))$, for all $\alpha \in \mathfrak{g}$ and $a \in \mathcal{A}$. 
  \begin{align*} 
    \hat{\phi}(\alpha(a))-\hat{\psi}(\alpha)(\hat{\phi}(a)) &=\phi^{\prime}(\phi(\alpha(a)))-\psi^{\prime}(\psi(\alpha))(\phi^{\prime}(\phi(a)))\\&=\phi^{\prime}(\psi(\alpha)(\phi(a)))-\psi^{\prime}(\psi(\alpha))(\phi^{\prime}(\phi(a)))\\&=\psi^{\prime}(\psi(\alpha)(\phi^{\prime}(\phi(a)))-\psi^{\prime}(\psi(\alpha))(\phi^{\prime}(\phi(a)))=0.
  \end{align*}
  
  (3) $\hat{\phi}(g(\alpha,\beta))=g^{\prime\prime}(\hat{\psi}(\alpha),\hat{\psi}(\beta))$, for $\alpha,\beta \in \mathcal{A}$. 
  \begin{align*} 
    \hat{\phi}(g(\alpha,\beta))-g^{\prime\prime}(\hat{\psi}(\alpha),\hat{\psi}(\beta)) &=\phi^{\prime}(\phi(g(\alpha,\beta)))-g^{\prime\prime}(\psi^{\prime}(\psi(\alpha)),\psi^{\prime}(\psi(\beta)))\\&=\phi^{\prime}(g^{\prime}(\psi(\alpha),\psi(\beta)))-g^{\prime\prime}(\psi^{\prime}(\psi(\alpha)),\psi^{\prime}(\psi(\beta)))\\&=g^{\prime\prime}(\psi^{\prime}(\psi(\alpha)),\psi^{\prime}(\psi(\beta)))-g^{\prime\prime}(\psi^{\prime}(\psi(\alpha)),\psi^{\prime}(\psi(\beta)))=0.
  \end{align*} 
  
  (4) $\phi^{\prime}(\phi(\mathcal{A}_\text{fin}))\subseteq \mathcal{A}^{\prime\prime}_{\text{fin}} $, since, $\phi(\mathcal{A}_\text{fin})\subseteq \mathcal{A}^{\prime}_{\text{fin}} $ and $\phi^{\prime}(\mathcal{A}^{\prime}_\text{fin})\subseteq \mathcal{A}^{\prime\prime}_{\text{fin}} $.\\
  Therefore, $(\hat{\phi},\hat{\psi})$ is a homomorphism of
  K\"ahler-Poisson algebras.
\end{proof}

\begin{remark}
  Note that, the composition of K\"ahler-Poisson algebra morphism is
  associative and the identity
  $(\operatorname{id}_\mathcal{A},\operatorname{id}_\mathfrak{g})$ of
  $\mathcal{K}=(\mathcal{A},g,\{x^1,...,x^m\})$ with
  $\operatorname{id}_\mathcal{A}:\mathcal{A}\rightarrow \mathcal{A}$
  and
  $\operatorname{id}_\mathfrak{g}:\mathfrak{g}\rightarrow
  \mathfrak{g}$ is a morphism of K\"ahler-Poisson algebras.
\end{remark}

\begin{remark}
\label{remark3.6}
An isomorphism of K\"ahler-Poisson algebras is a morphism
$(\phi,\psi)$ of K\"ahler-Poisson algebras such that $\phi$ is a
Poisson algebra isomorphism and
$\phi(\mathcal{A}_\text{fin})= \mathcal{A}^{\prime}_{\text{fin}}
$. Observe that, $\psi$ can always be constructed from $\phi$. When
two K\"ahler-Poisson algebras $\mathcal{K}$ and $\mathcal{K}^{\prime}$
are isomorphic, we write $\mathcal{K}\cong \mathcal{K}^{\prime}$.
\end{remark}

Before continuing, we need to introduce some notation. Let
$(\mathcal{A},\{.,.\})$ and $(\mathcal{A}^{\prime},\{.,.\}^{\prime})$
be Poisson algebras and let $x^i\in \mathcal{A}$ for $i=1,...,m.$ If
$p \in \mathcal{A}$ is a polynomial in $\{x^1,...,x^m\}$ then, using
Leibniz rule, one computes \begin{equation}
\label{equation}
\{p,a\}=\frac{\partial {p}}{\partial {x^i}}\{x^i,a\}
\end{equation}
where $\frac{\partial{p}}{\partial{x^i}}$ denotes the formal
derivative of the polynomial $p$ with respect to the variable
$x^i$. Note that, in general, $\frac{\partial{p}}{\partial{x^i}}$ is
itself not well-defined in the algebra, since there might exist
several different (but equivalent) representations of $p$ as a
polynomial in $x^1,...,x^m$, and the formal derivative then yields
several, possibly non-equivalent, elements of the algebra. However,
the combination in \eqref{equation} is always well-defined, and gives
the same result for all representations of $p$.

Given a matrix $M=(m_{ij})$ over $\mathcal{A}$, we set
$\phi(M)=(\phi(m_{ij}))$.  Given a morphism
$(\phi,\psi):(\mathcal{A},g,\{x^1,...,x^m\})\rightarrow
(\mathcal{A}^{\prime},g^{\prime},\{y^1,...,{y^m}^{\prime}\})$, it will
be convenient to introduce the notation
\begin{center}
  ${A^i}_\alpha=\frac{\partial \phi(x^i)}{\partial y^\alpha}$
\end{center}
(keeping in mind that this is not well-defined by itself); recall that
if $(\phi,\psi)$ is a morphism of K\"ahler-Poisson algebras, then
$\phi(\mathcal{A}_\text{fin})\subseteq
\mathcal{A}^{\prime}_{\text{fin}}$, ensuring that $\phi(x^i)$ is
indeed a polynomial in $y^1,...,{y^m}^{\prime}$. This notation allows
us to write \begin{center}
  $\phi(\{x^i,x^j\})=\{\phi(x^i),\phi(x^j)\}^{\prime}={A^i}_\alpha\{y^\alpha,y^\beta\}^\prime{A^j}_\beta$
\end{center}
in matrix notation: \begin{center}
$\phi(\mathcal{P})=A\mathcal{P}^\prime A^T$,
\end{center}
where $\mathcal{P}=(\{x^i,x^j\})$ and $\mathcal{P}^\prime=(\{y^\alpha,y^\beta\}^\prime)$.

Given two K\"ahler-Poisson algebras
$\mathcal{K}=(\mathcal{A},g,\{x^1,...,x^m\})$ and
$\mathcal{K}^{\prime}=(\mathcal{A}^{\prime},g^{\prime},\{y^1,...,{y^m}^{\prime}\})$,
we would like to understand when they are isomorphic. In the
following, we shall consider a number of examples in order to explore
when K\"ahler-Poisson algebras are isomorphic. From now on, we
consider only finitely generated Poisson algebras that have been
properly localized (cf. Example \ref{example.two elements}) to allow
the construction of a K\"ahler-Poisson algebra. We begin by taking
finitely generated algebras $\mathcal{A}$ and $\mathcal{A}^{\prime}$,
where $\mathcal{A}=\mathcal{A}^{\prime}$ and we consider different set
of generators for the K\"ahler-Poisson algebra structures of
$\mathcal{A}$ and $\mathcal{A}^{\prime}$. Let us start with the
following example.

\begin{example} 
\label{example4.0.5}
Let $(\mathcal{A},\{.,.\})$ be a Poisson algebra generated by two elements $x$ and $y$. From Example ~\ref{example.two elements} we know that
$\mathcal{K}=(\mathcal{A}[(\{x,y\}^2\det(g))^{-1}],g,\{x,y\})$
is a K\"ahler-Poisson algebra for an arbitrary symmetric matrix
\begin{center} $g= \begin{pmatrix}
    g_{11} & g_{12} \\
    g_{12} & g_{22}\\
    
  \end{pmatrix} $,  $\det(g)\neq 0$ with $\eta=\big(\{x,y\}^{2}\det (g)\big)^{-1}$.
\end{center}   
Consider $\mathcal{K}^{\prime}=(\mathcal{A}[(\{x,y\}^2\det(g))^{-1}],h,\{x+y,x-y\})$, with a symmetric matrix \begin{center}
  $h= \begin{pmatrix}
    h_{11} & h_{12}   \\
    h_{12} & h_{22}   \\
  \end{pmatrix} $, $\det(h)\neq 0$.
\end{center} 
We have
\begin{align*}  \mathcal{P}^{\prime}= \begin{pmatrix}
    \{x+y,x+y\} & \{x+y,x-y\} \\
    \{x+y,x-y\} &\{x-y,x-y\}\\
  \end{pmatrix}=\begin{pmatrix}
    0 & -2\{x,y\} \\
    2\{x,y\} &0\\
  \end{pmatrix}. \end{align*} 
It is easy to check that for the symmetric matrix $h$ 
we obtain   
\begin{equation}
\label{equation4.0.5}
\mathcal{P}^{\prime}h\mathcal{P}^{\prime}h\mathcal{P}^{\prime}=-4\{x,y\}^2 \det(h)\mathcal{P}^{\prime},
\end{equation} 
giving $\eta^\prime=(4\{x,y\}^2 \det(h))^{-1}$.  For each metric $g$
in the definition of $\mathcal{K}$, we find a suitable matrix $h$ such
that $\mathcal{K}\cong \mathcal{K}^{\prime}$.  From
\begin{center}
  $\mathfrak{g}=\{a_1\{x,\cdot\}+a_2\{y,\cdot\}:a_1,a_2\in \mathcal{A}\}$,
\end{center}
and
\begin{align*}
  \mathfrak{g}^{\prime}&=\{a_1\{x+y,\cdot\}+a_2\{x-y,\cdot\}:a_1,a_2\in \mathcal{A}\}\\&=\{a_1\{x,\cdot\}+a_1\{y,\cdot\}+a_2\{x,\cdot\}-a_2\{y,\cdot\}: a_1,a_2 \in \mathcal{A}\}\\&=\{(a_1+a_2)\{x,\cdot\}+(a_1-a_2)\{y,\cdot\}: a_1,a_2\in \mathcal{A}\},
\end{align*}
we see that $\mathfrak{g}=\mathfrak{g}^{\prime}$.  Since we require
that $\mathcal{K}\cong \mathcal{K}^{\prime}$, we define maps
$\phi:\mathcal{A}\rightarrow \mathcal{A}$ and
$\psi:\mathfrak{g}\rightarrow \mathfrak{g}$ satisfying properties
(1-4).  We choose $\phi=id$ , $\psi=id$ and find a suitable choice of
matrix $h$ yielding that $(\phi,\psi)$ is an isomorphism of K\"ahler
Poisson algebras. Now, we check that properties (1-4) in Definition
\ref{definition} are satisfied.

$(1)$   $\phi(a)\psi(\alpha)=a\alpha=\psi(a\alpha)$

$(2)$ $\psi(\alpha)(\phi(a))=\alpha(a)=\phi(\alpha(a))$

$(3)$ To show that $g(\alpha,\beta)=h(\alpha,\beta)$, we
start from the left hand side 
\begin{align*}
  \phi(g(\alpha,\beta))&=\alpha(x^i)g_{ij}\beta(x^j)\\&= \alpha(x)g_{11}\beta(x)+\alpha(x)g_{12}\beta(y)+\alpha(y)g_{21}\beta(x)+\alpha(y)g_{22}\beta(y).
\end{align*} 
From the right hand side we get 
\begin{align*}
  h(\alpha,\beta) &= \alpha(x^i)h_{ij}\beta(x^j)\\&=\alpha(x+y)h_{11}\beta(x+y)+\alpha(x+y)h_{12}\beta(x-y)+\alpha(x-y)h_{21}\beta(x+y)\\& \quad+\alpha(x-y)h_{22}\beta(x-y)
  \\&=\alpha(x)h_{11}\beta(x)+\alpha(x)h_{11}\beta(y)+\alpha(y)h_{11}\beta(x)+\alpha(y)h_{11}\beta(y)+\alpha(x)h_{12}\beta(x)+\\& \quad -\alpha(x)h_{12}\beta(y)+\alpha(y)h_{12}\beta(x)-\alpha(y)h_{12}\beta(y)+\alpha(x)h_{21}\beta(x)+\alpha(x)h_{21}\beta(y)\\&\quad-\alpha(y)h_{21}\beta(x)-\alpha(y)h_{21}\beta(y)+\alpha(x)h_{22}\beta(x)-\alpha(x)h_{22}\beta(y)-\alpha(y)h_{22}\beta(x)\\&\quad+\alpha(y)h_{22}\beta(y)\\&=\alpha(x)h_{11}\beta(x)+\alpha(x)h_{11}\beta(y)+\alpha(y)h_{11}\beta(x)+\alpha(y)h_{11}\beta(y)+2\alpha(x)h_{12}\beta(x)\\& \quad -2\alpha(y)h_{12}\beta(y)+\alpha(x)h_{22}\beta(x)-\alpha(x)h_{22}\beta(y)-\alpha(y)h_{22}\beta(x)+\alpha(y)h_{22}\beta(y).
\end{align*}
We conclude that
\begin{align} 
  &\begin{aligned}
    \label{aligned1}
    \alpha(x)\beta(x): h_{11}+2h_{12}+h_{22}=g_{11} 
  \end{aligned}\\
  &\begin{aligned}
    \label{aligned2}
    \alpha(x)\beta(y): h_{11}-h_{22}=g_{12}
  \end{aligned}\\
  &\begin{aligned}
    \alpha(y)\beta(x): h_{11}-h_{22}=g_{21}
  \end{aligned}\\
  &\begin{aligned}
    \label{aligned4}
    \alpha(y)\beta(y):  h_{11}-2h_{12}+h_{22}=g_{22}.
  \end{aligned}
\end{align} 
From \eqref{aligned4} we obtain, $ h_{11}=2h_{12}-h_{22}+g_{22} $ and
setting this in \eqref{aligned1} we get \begin{center}
  $ 2h_{12}-h_{22}+g_{22}+2h_{12}+h_{22}=g_{11}$\\$
  4h_{12}=g_{11}-g_{22}$ \\$h_{12}=\frac{g_{11}-g_{22}}{4}$.
\end{center}
From  \eqref{aligned2} we get  $h_{22}=h_{11}-g_{12}$,
which in \eqref{aligned1} gives \begin{center}
  $h_{11}+2(\frac{g_{11}-g_{22}}{4})+h_{11}-g_{12} =g_{11}$\\
  $2h_{11}+(\frac{g_{11}-g_{22}}{2})-g_{12} =g_{11}$\\
  $4h_{11}+g_{11}-g_{22} =2(g_{11}+g_{12})$\\
  $4h_{11}=g_{11}+2g_{12}+g_{22}$\\
  $h_{11}=\frac{g_{11}+2g_{12}+g_{22}}{4}$,
\end{center} 
which implies that
\begin{align*}
  h_{22}=h_{11}-g_{12}=\frac{g_{11}+2g_{12}+g_{22}}{4}-g_{12}=\frac{g_{11}-2g_{12}+g_{22}}{4}.
\end{align*}			
Now, the symmetric matrix $h$ becomes
\begin{center}
  $h= \begin{pmatrix}
    \frac{g_{11}+2g_{12}+g_{22}}{4} &\frac{g_{11}-g_{22}}{4}   \\
    \frac{g_{11}-g_{22}}{4} & \frac{g_{11}-2g_{12}+g_{22}}{4}   \\
  \end{pmatrix} $, 
\end{center}
giving
\begin{align*}
  \det(h)&=\frac{1}{16}\Big[(g_{11}+2g_{12}+g_{22})(g_{11}-2g_{12}+g_{22})-(g_{11}-g_{22})^2\Big]\\&=\frac{1}{16}\Big[(g_{11})^2-2g_{11}g_{12}+g_{11}g_{22}+2g_{11}g_{12}-4(g_{12})^2+2g_{12}g_{22}+g_{11}g_{22}\\&\quad-2g_{12}g_{22}+(g_{22})^2-(g_{11})^2-(g_{22})^2+2g_{11}g_{22}\Big]\\&=\frac{1}{16}(4g_{11}g_{22}-4(g_{12})^2)=\frac{1}{4}(g_{11}g_{22}-(g_{12})^2)=\frac{1}{4}\det(g).
\end{align*}
Inserting $\det(h)$ in equation ~\ref{equation4.0.5} we get
\begin{align*}
  \mathcal{P}^{\prime}h\mathcal{P}^{\prime}h\mathcal{P}^{\prime}=-4\{x,y\}^2 \frac{1}{4}\det(g)\mathcal{P}^{\prime}=-\{x,y\}^2\det(g)\mathcal{P}^{\prime},
\end{align*} giving $\eta^{\prime}=(\{x,y\}^2\det(g))^{-1}$. Therefore $\eta^\prime=\eta$.
We conclude that $g(\alpha,\beta)=h(\alpha,\beta)$.

$(4)$ It is easy to see that $\phi(\mathcal{A}_\text{fin})\subseteq \mathcal{A}^{\prime}_{\text{fin}} $.\\
Hence, $\mathcal{K}\cong \mathcal{K}^{\prime}$ if
\begin{align*}
  h=
  \begin{pmatrix}
    \frac{g_{11}+2g_{12}+g_{22}}{4} &\frac{g_{11}-g_{22}}{4}   \\
    \frac{g_{11}-g_{22}}{4} & \frac{g_{11}-2g_{12}+g_{22}}{4}   \\
  \end{pmatrix}. 
\end{align*}
\end{example}

Note that the above example extends to the case where we choose more
(dependent) generators of the algebra, giving many possible
presentations of the same K\"ahler-Poisson algebra.  Next, let us
explore the case when we choose a different number of generators for
the same algebra.
\begin{example} 
  Let $(\mathcal{A},\{\cdot,\cdot\})$ be a Poisson algebra generated
  by two elements $x$ and $y$. Consider the K\"ahler-Poisson algebra
  $\mathcal{K}=(\mathcal{A}[(\{x,y\}^2\det(g))^{-1}],g,\{x,y\})$, with
  an arbitrary symmetric matrix
  \begin{center} $g= \begin{pmatrix}
      g_{11} & g_{12} \\
      g_{12} & g_{22}\\
      
    \end{pmatrix}$, $\det(g)\neq 0$ and 
    $\eta=\big(\{x,y\}^{2}\det (g)\big)^{-1}$.
  \end{center}   
  Let $\mathcal{K}^{\prime}=(\mathcal{A}[(\{x,y\}^2\det(g))^{-1}],h,\{x,y,x\})$ with
  \begin{center} 
    $\mathcal{P}^{\prime}= \begin{pmatrix}
      0 & \lambda & 0  \\
      -\lambda &0 & -\lambda \\
      0 & \lambda & 0 \\
    \end{pmatrix} $, 	
  \end{center}
  where $\lambda=\{x,y\}$.  It is easy to check that for the symmetric
  matrix h
  \begin{center}
    $h= \begin{pmatrix}
      \frac{1}{4}g_{11} & \frac{1}{2}g_{12} & \frac{1}{4}g_{11}  \\
      \frac{1}{2}g_{12} & g_{22} & \frac{1}{2}g_{12}  \\
      \frac{1}{4}g_{11} & \frac{1}{2}g_{12} & \frac{1}{4}g_{11}  \\
    \end{pmatrix} $, 
  \end{center}
  one obtains
  $\mathcal{P}^{\prime}h\mathcal{P}^{\prime}h\mathcal{P}^{\prime}=-\{x,y\}^2
  \det(g)\mathcal{P}^{\prime}$, giving
  $\eta^{\prime}=(\{x,y\}^2\det(g))^{-1}$. We conclude that
  $\eta^{\prime}=\eta$.  Now, we show that
  $\mathcal{K}\cong \mathcal{K}^{\prime}$. From
  \begin{center}
    $\mathfrak{g}=\{a_1\{x,\cdot\}+a_2\{y,\cdot\}:a_1,a_2\in \mathcal{A}\}$
  \end{center}
  and
  \begin{align*}
    \mathfrak{g}^{\prime}&=\{a_1\{x,\cdot\}+a_2\{y,\cdot\}+a_3\{x,\cdot\}:a_1,a_2,a_3\in \mathcal{A}\}\\&=\{(a_1+a_3)\{x,\cdot\}+a_2\{y,\cdot\}: a_1,a_2,a_3 \in \mathcal{A}\},
  \end{align*}
  we see that $\mathfrak{g}=\mathfrak{g}^{\prime}$.  We define maps
  $\phi:\mathcal{A}\rightarrow \mathcal{A}$ and
  $\psi:\mathfrak{g}\rightarrow \mathfrak{g}$ by choosing $\phi=id$
  and $\psi=id$, and we find a suitable choice of matrix $h$ such that
  $(\phi,\psi)$ is an isomorphism of K\"ahler Poisson algebras. We
  again check properties (1-4) in Definition \ref{definition} are
  satisfied.
	
  $(1)$ $\phi(a)\psi(\alpha)=a\alpha=\psi(a\alpha)$.
  
  $(2) $ $\psi(\alpha)(\phi(a))=\alpha(a)=\phi(\alpha(a))$.
  
  $(3)$ To show that $\phi(g(\alpha,\beta))=h(\psi(\alpha),\psi(\beta))$, we
  start from the right hand side 
  \begin{align}\label{eq.3.7}
    h(\psi(\alpha),\psi(\beta)) & \nonumber=h(\alpha,\beta)\\& =\nonumber \alpha(x^i)h_{ij}\beta(x^j)\\&=\nonumber\alpha(x)h_{11}\beta(x)+\alpha(x)h_{12}\beta(y)+\alpha(y)h_{13}\beta(x)+\alpha(y)h_{21}\beta(x)+\alpha(y)h_{22}\beta(y)
    \\& \nonumber \quad+\alpha(y)h_{23}\beta(x)+\alpha(x)h_{31}\beta(x)+\alpha(x)h_{32}\beta(y)+\alpha(x)h_{33}\beta(x)
    \\&=\nonumber\frac{1}{4}\alpha(x)g_{11}\beta(x)+\frac{1}{2}\alpha(x)g_{12}\beta(y)+\frac{1}{4}\alpha(x)g_{11}\beta(x)+\frac{1}{2}\alpha(y)g_{12}\beta(x)+\alpha(y)g_{22}\beta(y)\\& \nonumber \quad+\frac{1}{2}\alpha(y)g_{12}\beta(x)+\frac{1}{4}\alpha(x)g_{11}\beta(x)+\frac{1}{2}\alpha(x)g_{12}\beta(y)+\frac{1}{4}\alpha(x)g_{11}\beta(x)
    \\&=\alpha(x)g_{11}\beta(x)+\alpha(x)g_{12}\beta(y)+\alpha(y)g_{12}\beta(x)+\alpha(y)g_{22}\beta(y).
  \end{align}
  From the left hand side we get
  \begin{align*}
    \phi(g(\alpha,\beta))&=\alpha(x^i)g_{ij}\beta(x^j)\\&= \alpha(x)g_{11}\beta(x)+\alpha(x)g_{12}\beta(y)+\alpha(y)g_{21}\beta(x)+\alpha(y)g_{22}\beta(y),
  \end{align*}
  and we conclude that $\phi(g(\alpha,\beta))=h(\psi(\alpha),\psi(\beta))$.
  
  $(4)$ It is easy to see that $\phi(\mathcal{A}_\text{fin})\subseteq \mathcal{A}^{\prime}_{\text{fin}} $.\\
  Therefore, $\mathcal{K}\cong \mathcal{K}^{\prime}$ for the matrix
  $h$ with entries as in \eqref{eq.3.7}.
\end{example}

The above examples show that two K\"ahler-Poisson algebras can be
isomorphic when considering different set of generators of the same
finitely generated Poisson algebra.

\subsection{Properties of isomorphisms}

We are interested in studying properties of more general isomorphisms
for K\"ahler-Poisson algebras. Let Let
$\mathcal{K}=(\mathcal{A},g,\{x^1,...,x^m\})$ and
$\mathcal{K}^{\prime}=(\mathcal{A}^{\prime},g^{\prime},\{y^1,...,{y^m}^{\prime}\})$
be K\"ahler-Poisson algebras, and assume that there exists a Poisson
algebra isomorphism
$\phi:(\mathcal{A},\{.,.\})\rightarrow
(\mathcal{A}^{\prime},\{.,.\}^{\prime})$. When does there exist a map
$\psi:\mathfrak{g}\rightarrow \mathfrak{g}^{\prime}$ such that
$(\phi,\psi)$ is an isomorphism of K\"ahler-Poisson algebras?. The
following result provides an answer to this question.

\begin{proposition}
\label{proposition4}
Let $\mathcal{K}=(\mathcal{A},g,\{x^1,...,x^m\})$ and
$\mathcal{K}^{\prime}=(\mathcal{A}^{\prime},g^{\prime},\{y^1,...,{y^m}^{\prime}\})$
be K\"ahler-Poisson algebras. Then $\mathcal{K}$ and $\mathcal{K}^{\prime}$
are isomorphic if and only if there exists a Poisson algebra
isomorphism $\phi:\mathcal{A}\rightarrow \mathcal{A}^{\prime}$ such
that $\phi(\mathcal{A}_\text{fin})= \mathcal{A}^{\prime}_\text{fin} $,
and
\begin{align}
  \label{properties of isomorphisms}
  \mathcal{P}^{\prime}g^{\prime}\mathcal{P}^{\prime}&=\mathcal{P}^{\prime}A^T\phi(g)A\mathcal{P}^{\prime},
\end{align} 
where ${A^i}_\alpha=\frac{\partial \phi(x^i)}{\partial y^\alpha}$ and
$(\mathcal{P}^{\prime})^{\alpha\beta}=\{y^\alpha,y^\beta\}^{\prime}$.
\end{proposition}

\begin{proof}
  Assume that $\mathcal{K}\cong \mathcal{K}^{\prime}$. Then we need to
  show that
  \begin{align*}
    \mathcal{P}^{\prime}g^{\prime}\mathcal{P}^{\prime}&=\mathcal{P}^{\prime}A^T\phi(g)A\mathcal{P}^{\prime}.
  \end{align*}
  Let us start by computing $\phi(\mathcal{P}^{ij})$ 
  \begin{align*}
    \phi(\mathcal{P}^{ij})&=\phi(\{x^i,x^j\})=\{\phi(x^i),\phi(x^j)\}^{\prime}=\frac{\partial \phi(x^i)}{\partial y^\alpha}\{y^\alpha,\phi(x^i)\}^{\prime}\\&=\frac{\partial \phi(x^i)}{\partial y^\alpha}\{y^\alpha,y^\beta\}^{\prime}\frac{\partial \phi(x^j)}{\partial y^\beta}={A^i}_\alpha (\mathcal{P}^{\prime})^{\alpha \beta} {A^j}_\beta.
  \end{align*}
  Now, in order to prove \eqref{properties of isomorphisms}, we let $\gamma_i=\{x^i,.\}$ and $\gamma_j=\{x^j,.\}$  then
  \begin{align*}
    \phi(g(\gamma_i,\gamma_j)) &=\phi(\{x^i,x^k\}g_{kl}\{x^j,x^l\}) =\phi(\mathcal{P}^{ik}g_{kl}\mathcal{P}^{jl}).
  \end{align*}
  From $(2)$ in Definition \ref{definition} it follows that
  \begin{center}
    $\psi(\gamma_i)(b)=\phi(\{x^i,\phi^{-1}(b)\})=\{\phi(x^i),b\}^{\prime}$,
  \end{center}
  which implies that, $\psi(\gamma_i)=\{\phi(x^i),\cdot\}^{\prime}$ and in the same way we get $\psi(\gamma_j)=\{\phi(x^j),\cdot\}^{\prime}$.
  
  Furthermore,
  \begin{align*}
    \phi(\mathcal{P}^{ik}g_{kl}\mathcal{P}^{jl})&=\{\phi(x^i),y^\alpha\}^{\prime}{A^k}_\alpha \phi (g_{kl})\{\phi(x^j),y^\beta\}^{\prime} {A^l}_\beta,
  \end{align*}
  and 
  \begin{align*}
    g^{\prime}(\psi(\gamma_i),\psi(\gamma_j))&=\{\phi(x^i),y^\alpha\}^{\prime}g^{\prime}_{\alpha \beta}\{\phi(x^j),y^\beta\}^{\prime}.
  \end{align*}
  From $(3)$ in Definition \ref{definition} one obtains
  \begin{equation}
    \label{equation:*}
    \{\phi(x^i),y^\alpha\}^{\prime}({A^k}_\alpha \phi (g_{kl}){A^l}_\beta-g^{\prime}_{\alpha \beta})\{\phi(x^j),y^\beta\}^{\prime}=0.
  \end{equation}
  Note that if  $\{\phi(x^i),y^\beta\}C_\beta=0$, for $C_\beta \in \mathcal{A}^{\prime}$, then
  \begin{align*}
    \{y^\alpha,y^\beta\}^{\prime}C_\beta&= \phi(\{\phi^{-1}(y^\alpha),\phi^{-1}(y^\alpha)\})C_\beta=\phi\Big( \frac{\partial \phi^{-1}(y^\alpha)}{\partial x^i}\{x^i,\phi^{-1}(y^\beta)\}\Big)C_\beta\\&=\phi\Big(\frac{\partial \phi^{-1}(y^\alpha)}{\partial x^i}\Big)\phi\Big(\{x^i,\phi^{-1}(y^\beta)\}\Big)C_\beta=\phi\Big(\frac{\partial \phi^{-1}(y^\alpha)}{\partial x^i}\Big)\{\phi(x^i),y^\beta\}^{\prime}C_\beta=0.
  \end{align*}
  Therefore, equation \eqref{equation:*} yields
  \begin{center}
    $\{y^\gamma,y^\alpha\}^\prime({A^k}_\alpha \phi (g_{kl}){A^l}_\beta-g^{\prime}_{\alpha \beta})\{\phi(x^j),y^\beta\}^\prime=0$, 	
  \end{center}
  and furthermore 
  \begin{center}	
    $\{y^\gamma,y^\alpha\}^{\prime}({A^k}_\alpha \phi (g_{kl}){A^l}_\beta-g^{\prime}_{\alpha \beta})\{y^\delta,y^\beta\}^\prime=0$.
  \end{center}
  In matrix notation this becomes 
  \begin{equation}\label{eq4.2.2}
    \mathcal{P}^{\prime}A\phi(g)A^T\mathcal{P}^{\prime}=\mathcal{P}^{\prime}g^{\prime}\mathcal{P}^{\prime}.
  \end{equation}
 
  To prove the converse, we assume that
  $\phi:\mathcal{A}\rightarrow \mathcal{A}^{\prime}$
  ($\phi(a)=a^{\prime}$) is an isomorphism such that
  $\phi(\mathcal{A}_\text{fin})=\mathcal{A}^\prime_{\text{fin}}$ and
  \eqref{eq4.2.2} holds. First, we need to define the map
  $\psi$. Since $\phi$ is an isomorphism of Poisson algebras, we may
  define $\psi$ as:
  \begin{center}
    $\psi(\gamma_i)(a^{\prime}):=\phi(\alpha(\phi^{-1}(a^{\prime})))$,
  \end{center}
  which clearly fulfills:                                                                          
  \begin{center}
    $\phi(\gamma_i(a))=\psi(\gamma_i)(\phi(a))$.
  \end{center}
  With a slight abuse of notation, let us show that
  $\psi(\alpha) \in \mathfrak{g}^{\prime}$ for
  $\alpha \in \mathfrak{g} $.  Writing $\alpha=a_i\{b^i,.\}$ and
  $\beta=b_i\{b^i,.\}$ as inner derivations in $\mathfrak{g}$ one
  obtains
  \begin{align*}
    \psi(\alpha)(a^{\prime})=\phi(\alpha(\phi^{-1}(a^{\prime}))) =\phi(a_i\{b^i,\phi^{-1}(a^{\prime})\}) =\phi(a_i)\{\phi(b^i),a^{\prime}\}^{\prime},
  \end{align*}
  which implies that
  $\psi(\alpha)=\phi(a_i)\{\phi(b^i),.\}^{\prime}\in\mathfrak{g}^{\prime}$.
  Similarly, one obtains
  $\psi(\beta)=\phi(b_i)\{\phi(b^i),.\}^{\prime} \in
  \mathfrak{g}^{\prime}$.  Secondly, we show that $\psi$ is a Lie
  algebra isomorphism:
  \begin{align*}
    \text{I)}    \qquad [\psi(\alpha),\psi(\beta)](a^{\prime})&=\psi(\alpha)\big(\psi(\beta)(a^{\prime})\big)-\psi(\beta)\big(\psi(\alpha)(a^{\prime})\big) \\&=\psi(\alpha)\big(\phi(\beta(\phi^{-1}(a^{\prime})))\big)-\psi(\beta)\big(\phi(\alpha(\phi^{-1}(a^{\prime})))\big) 
    \\&=\phi\big(\alpha(\beta(\phi^{-1}(a^{\prime})))\big)-\phi\big(\beta(\alpha(\phi^{-1}(a^{\prime})))\big) \\&=\phi\big(\alpha(\beta(\phi^{-1}(a^{\prime})))-\beta(\alpha(\phi^{-1}(a^{\prime})))\big)
    \\&=\phi\big([\alpha,\beta](\phi^{-1}(a^{\prime})\big) \\&=\psi([\alpha,\beta])(a^{\prime}).
  \end{align*}

  \begin{align*} 
    \text{II)} \quad \psi(\alpha)(a^{\prime})+\psi(\beta)(a^{\prime})&=\phi\big(\alpha(\phi^{-1}(a^{\prime}))\big) +\phi\big(\beta(\phi^{-1}(a^{\prime}))\big)= \phi\big(\alpha(\phi^{-1}(a^{\prime}) +\beta(\phi^{-1}(a^{\prime}))\big)\\&= \phi\big((\alpha+\beta)(\phi^{-1}(a^{\prime}))\big)
    =\psi(\alpha+\beta)(a^{\prime}).
  \end{align*}
  Therefore $\psi$ is a Lie algebra homomorphism. Now, we show that
  the map $\psi^{-1}$ defined by
  \begin{center}
    $\psi^{-1}(\alpha^{\prime})(a):=\phi^{-1}(\alpha^{\prime}(\phi(a)))$,
  \end{center}
  for all $a \in \mathcal{A}$ and
  $\alpha^{\prime} \in \mathfrak{g}^{\prime}$ is indeed the inverse of $\psi$:
  \begin{align*}
    \psi^{-1}\big(\psi(\alpha)\big)(a)=\phi^{-1}\Big(\psi(\alpha)\big(\phi(a)\big)\Big) = \phi^{-1}\Big(\phi\big(\alpha(\phi^{-1}\phi(a))\big)\Big) = \alpha(a)
  \end{align*}
  and
  \begin{align*}
    \psi\big(\psi^{-1}(\alpha^{\prime})\big)(a^{\prime})=\phi\Big(\psi^{-1}(\alpha^{\prime})\big(\phi^{-1}(a^{\prime})\big)\Big) = \phi\Big(\phi^{-1}\big(\alpha^{\prime}(\phi\phi^{-1}(a^{\prime}))\big)\Big) = \alpha^{\prime}(a^{\prime}).
  \end{align*}
  Finally, we show that $(\phi,\psi)$ is a morphism of
  K\"ahler-Poisson algebras.

  $(1)$  $\phi(a)\psi(\alpha)(a^{\prime})=\phi(a)\phi(\alpha(\phi^{-1}(a^{\prime}))) =\phi(a\alpha(\phi^{-1}(a^{\prime})))=\psi(a\alpha)(a^{\prime}).$ 
  
  $(2)$ $\psi(\alpha)(\phi(a))=\phi(\alpha(\phi^{-1}(\phi(a)))) =\phi(\alpha(a)).$
  
  $(3)$ For $\alpha=\alpha_i\{x^i,.\}$ and $\beta=\beta_i\{x^i,.\}$ one gets
  \begin{align*}
    \phi(g(\alpha,\beta))&=\phi(\alpha_i\{x^i,x^k\}g_{kl}\beta_j\{x^j,x^l\}) =\phi(\alpha_i)\phi(\{x^i,x^k\})\phi(g_{kl})\phi(\beta_j)\phi(\{x^j,x^l\})
    \\&=\phi(\alpha_i)\phi(\mathcal{P}^{ik})\phi(g_{kl})\phi(\beta_j)\phi(\mathcal{P}^{jl}) =\phi(\alpha_i)(A\mathcal{P}^{\prime}A^T)^{ik}\phi(g_{kl})\phi(\beta_j)(A\mathcal{P}^{\prime}A^T)^{jl}
    \\&=\phi(\alpha_i)\phi(\beta_j)\big(A\mathcal{P}^{\prime}A^T\phi(g)A(-\mathcal{P}^{\prime})A^T\big)^{ij} =-\phi(\alpha_i)\phi(\beta_j)\big(A\mathcal{P}^{\prime}A^T\phi(g)A\mathcal{P}^{\prime}A^T\big)^{ij}.
  \end{align*}
  By using
  \begin{center}
    $\mathcal{P}^{\prime}g^{\prime}\mathcal{P}^{\prime}=\mathcal{P}^{\prime}A^T\phi(g)A\mathcal{P}^{\prime}$,
  \end{center}
  one obtains
  \begin{align*}
    \phi(g(\alpha,\beta))&=-\phi(\alpha_i)\phi(\beta_j)(A\mathcal{P}^{\prime}g^{\prime}\mathcal{P}^{\prime}A^T)^{ij}
                           =-\phi(\alpha_i)\phi(\beta_j){A^i}_\alpha \{y^\alpha,y^\beta\}g^{\prime}_{\beta\gamma} \{y^\gamma,y^\delta\}{A^j}_\delta \\&=-\phi(\alpha_i)\phi(\beta_j) \{\phi(x^i),y^\beta\}g^{\prime}_{\beta\gamma} \{y^\gamma,\phi(x^j)\} =\psi(\alpha)(y^\beta)g^{\prime}_{\beta\gamma}\psi(\beta)(y^\gamma) 
    \\&=g^{\prime}(\psi(\alpha),\psi(\beta)).
  \end{align*}
  Note that $(4)$ is true by assumption.
\end{proof}

As an illustration, Proposition~\ref{proposition4} gives us another
method to do the calculations of Example \ref{example4.0.5}. We have
seen in Example \ref{example4.0.5} that two K\"ahler-Poisson algebras
$\mathcal{K}=(\mathcal{A},g,\{x^1,...,x^m\})$ and
$\mathcal{K}^{\prime}=(\mathcal{A},g^{\prime},\{y^1,...,{y^m}^{\prime}\})$,
where $\mathcal{A}$ is a finitely generated algebra, can be isomorphic
when considering different sets of generators for the K\"ahler-Poisson
algebra structures of $\mathcal{A}$ and $\mathcal{A}^{\prime}$.

\begin{example}(Continuation of ~\ref{example4.0.5})
  Proposition~\ref{proposition4} tells us that if we set
  $h=A^T\phi(g)A$, then
  $\mathcal{P^\prime}h\mathcal{P^\prime}=\mathcal{P^\prime}A^T\phi(g)A\mathcal{P^\prime}$
  implying that $\mathcal{K}\cong \mathcal{K}^{\prime}$.  Let
  $y^1=x+y$ and $y^2=x-y$. Hence $x=\frac{1}{2}(y^1+y^2)$ and
  $y=\frac{1}{2}(y^1-y^2)$. We compute the matrix
  ${A^i}_\alpha=\frac{\partial x^i}{\partial y^\alpha}$, with
  $\phi=id$: \begin{center}
    ${A^1}_1=\frac{\partial x^1}{\partial y^1}=\frac{\partial
      (\frac{1}{2}(y^1+y^2))}{\partial y^1}=\frac{1}{2}$,
    ${A^1}_2=\frac{\partial x^1}{\partial y^2}=\frac{\partial
      (\frac{1}{2}(y^1+y^2))}{\partial y^2}=\frac{1}{2}$,
  \end{center}
  \begin{center}
    ${A^2}_1=\frac{\partial x^2}{\partial y^1}=\frac{\partial y}{\partial y^1}=\frac{\partial (\frac{1}{2}(y^1-y^2))}{\partial y^1}=\frac{1}{2}$,
    ${A^2}_1=\frac{\partial x^2}{\partial y^2}=\frac{\partial y}{\partial y^2}=\frac{\partial (\frac{1}{2}(y^1-y^2))}{\partial y^2}=-\frac{1}{2}$.
  \end{center}
  Therefore the matrix $A$ becomes
  \begin{center}
    $A=\begin{pmatrix}
      \frac{1}{2}& \frac{1}{2}\\
      \frac{1}{2}& -\frac{1}{2}\\
    \end{pmatrix}$
  \end{center}
  and
  \begin{align*}
    h&=\begin{pmatrix}
      \frac{1}{2}& \frac{1}{2}\\
      \frac{1}{2}& -\frac{1}{2}\\
    \end{pmatrix}
    \begin{pmatrix}
      g_{11}& g_{22}\\
      g_{21}& g_{22}\\
    \end{pmatrix}
    \begin{pmatrix}
      \frac{1}{2}& \frac{1}{2}\\
      \frac{1}{2}& -\frac{1}{2}\\
    \end{pmatrix}
    =\begin{pmatrix}
      \frac{1}{2}(g_{11}+g_{21})& \frac{1}{2}(g_{12}+g_{22})\\
      \frac{1}{2}(g_{11}-g_{21})& \frac{1}{2}(g_{12}-g_{22})\\
    \end{pmatrix}\begin{pmatrix}
      \frac{1}{2}& \frac{1}{2}\\
      \frac{1}{2}& -\frac{1}{2}\\
    \end{pmatrix}
    \\&=\begin{pmatrix}
      \frac{1}{4}(g_{11}+g_{21})+\frac{1}{4}(g_{12}+g_{22})& \frac{1}{4}(g_{11}+g_{21})-\frac{1}{4}(g_{12}+g_{22})\\
      \frac{1}{4}(g_{11}-g_{21})+\frac{1}{4}(g_{12}-g_{22})& \frac{1}{4}(g_{11}-g_{21})-\frac{1}{4}(g_{12}-g_{22})\\
    \end{pmatrix}
    \\&=\begin{pmatrix}
      \frac{1}{4}(g_{11}+2g_{12}+g_{22})& \frac{1}{4}(g_{11}-g_{22})\\
      \frac{1}{4}(g_{11}-g_{22})& \frac{1}{4}(g_{11}-2g_{12}+g_{22})\\
    \end{pmatrix},
  \end{align*}
  which agrees with the result in Example
  ~\ref{example4.0.5}. Also
  $\phi(\mathcal{A}_\text{fin})=\phi(\mathcal{A^\prime_{\text{fin}}})$,
  since $\phi=id_\mathcal{A}$. Therefore we can now use
  Proposition ~\ref{proposition4} to conclude that
  $\mathcal{K}\cong \mathcal{K}^{\prime}$.
\end{example}

In the above examples of isomorphism
$\mathcal{K}\cong \mathcal{K}^{\prime}$, we noted that
$\eta=\eta^\prime$. The next proposition shows that this is indeed
true in the generic case.

\begin{proposition}
  \label{proposition4.2.3}
  Let $\mathcal{K}=(\mathcal{A},g,\{x^1,...,x^m\})$ and
  $\mathcal{K}^{\prime}=(\mathcal{A}^{\prime},g^{\prime},\{y^1,...,{y^m}^{\prime}\})$
  be K\"ahler-Poisson algebras and let
  $(\phi,\psi):\mathcal{K}\rightarrow \mathcal{K}^{\prime}$ be an
  isomorphism of K\"ahler-Poisson algebras.  If
  \begin{align*}
    \eta\mathcal{P}g\mathcal{P}g\mathcal{P}=-\mathcal{P} \quad\text{and}\quad\eta^{\prime}\mathcal{P}^{\prime}g^{\prime}\mathcal{P}^{\prime}g^{\prime}\mathcal{P}^{\prime}=-\mathcal{P}^{\prime}
  \end{align*}
  then $(\phi(\eta)-\eta^{\prime})\mathcal{P}^\prime=0$.
\end{proposition}

\begin{proof}
  From Proposition~\ref{proposition4} we have
  $\mathcal{P}^{\prime}g^{\prime}\mathcal{P}^{\prime}=\mathcal{P}^{\prime}A^T\phi(g)A\mathcal{P}^{\prime}.$
  Starting from $\eta\mathcal{P}g\mathcal{P}g\mathcal{P}=-\mathcal{P}$
  and using that $\phi(\mathcal{P})=A\mathcal{P}^{\prime}A^T$ and
  $\mathcal{P}^{\prime}g^{\prime}\mathcal{P}^{\prime}=\mathcal{P}^{\prime}A^T\phi(g)A\mathcal{P}^{\prime}$,
  one has
  $-\phi(\mathcal{P})=\phi(\eta)\phi(\mathcal{P}g\mathcal{P}g\mathcal{P}).$
  Multiplying both sides by $\eta^{\prime}$, where
  $\eta^{\prime}\mathcal{P}^{\prime}g^{\prime}\mathcal{P}^{\prime}g^{\prime}\mathcal{P}^{\prime}=-\mathcal{P}^{\prime}$,
  we obtain
  \begin{align*}
    -\eta^{\prime}\phi(\mathcal{P})&=\phi(\eta)\phi(\mathcal{P}g\mathcal{P}g\mathcal{P})\eta^{\prime}
                                     =\phi(\eta)A\mathcal{P}^{\prime}A^T\phi(g)A\mathcal{P}^{\prime}A^T\phi(g)A\mathcal{P}^{\prime}A^T\eta^{\prime}
    \\&=\phi(\eta)A\mathcal{P}^{\prime}g^{\prime}\mathcal{P}^{\prime}A^T\phi(g)A\mathcal{P}^{\prime}A^T\eta^{\prime}
    =\phi(\eta)A\eta^{\prime}\mathcal{P}^{\prime}g^{\prime}\mathcal{P}^{\prime}g^{\prime}\mathcal{P}^{\prime}A^T
    =-\phi(\eta)A\mathcal{P}^{\prime}A^T.
  \end{align*}
  Hence, one obtains
  $\eta^{\prime}A\mathcal{P}^{\prime}A^T=\phi(\eta)A\mathcal{P}^{\prime}A^T,$
  which implies
  that$(\eta^{\prime}-\phi(\eta))\mathcal{P}^{\prime}=0,$ using the
  same argument as in the proof of Proposition ~\ref{proposition4}
\end{proof}

Thus, if at least one of
$(\mathcal{P^{\prime}})^{\alpha\beta}=\{y^\alpha,y^\beta\}$ is not a
zero divisor, then Proposition ~\ref{proposition4.2.3} implies that
$\phi(\eta)=\eta^\prime$.

\section{Subalgebras of K\"ahler-Poisson algebras}
As shown in Section 3, a morphism of K\"ahler-Poisson algebras is a
pair of maps $(\phi,\psi)$, with
$\phi:\mathcal{A}\rightarrow \mathcal{A}^{\prime}$ a Poisson algebra
homomorphism and $\psi:\mathfrak{g}\rightarrow \mathfrak{g}^{\prime}$
a Lie algebra homomorphism, satisfying certain conditions.  In this
section, we present some algebraic properties of K\"ahler-Poisson
algebras. In particular, we introduce subalgebras of K\"ahler-Poisson
algebras, which we define by using the concept of morphisms.

\begin{definition}
  \label{definition2}
  Let $\mathcal{K}=(\mathcal{A},g,\{x^1,...,x^m\})$ and
  $\mathcal{K}^{\prime}=(\mathcal{A}^{\prime},g^{\prime},\{y^1,...,{y^m}^{\prime}\})$
  be K\"ahler-Poisson algebras together with their corresponding
  modules of derivations $\mathfrak{g}$ and $\mathfrak{g}^{\prime}$,
  respectively.  $\mathcal{K}$ is a K\"ahler-Poisson subalgebra of
  $\mathcal{K}^{\prime}$ if:
  \begin{enumerate}[(A)]
  \item $\mathcal{A}$ is a Poisson subalgebra of $\mathcal{A}^{\prime}$,
  \item $(\operatorname{id}|_\mathcal{A},\operatorname{id}|_\mathfrak{g})$ is a homomorphism of K\"ahler-Poisson algebras,where $\operatorname{id}|_\mathcal{A}$ and $\operatorname{id}|_\mathfrak{g}$ denotes the identity maps restricted to $\mathcal{A}$ and $\mathfrak{g}$, respectively.
  \end{enumerate}
\end{definition}

Given two K\"ahler-Poisson algebras
$\mathcal{K}=(\mathcal{A},g,\{x^1,...,x^m\})$ and
$\mathcal{K}^{\prime}=(\mathcal{A}^{\prime},g^{\prime},\{y^1,...,{y^m}^{\prime}\})$,
we illustrate with examples when $\mathcal{K}$ is a K\"ahler-Poisson
subalgebra of $\mathcal{K}^{\prime}$, where $\mathcal{A}$ is a proper
Poisson subalgebra of a finitely generated algebra
$\mathcal{A}^{\prime}$. We consider different set of generators for
the K\"ahler-Poisson algebra structures of $\mathcal{A}$ and
$\mathcal{A}^{\prime}$. Note that the property (B) in Definition
\ref{definition2} determines the metric $g$ on $\mathcal{K}$.

\begin{example}
  Let $\mathcal{K}=(\mathcal{A},g,\{x,y\})$ and
  $\mathcal{K}^{\prime}=(\mathcal{A}^{\prime},g^{\prime},\{x,y,z\})$
  be K\"ahler-Poisson algebras, where $\mathcal{A}$ is a subalgebra of
  $\mathcal{A}^{\prime}$ generated by $\{x,y\}$ and
  $\mathcal{A}^{\prime}$ generated by $\{x,y,z\}$. Let
  $\{x,y\}=p(x,y)$ and $\{x,z\}=\{y,z\}=0$. Since $\mathcal{A}$ is a
  Poisson subalgebra of $\mathcal{A}^{\prime}$, to show that
  $\mathcal{K}$ is a K\"ahler-Poisson subalgebra of
  $\mathcal{K}^{\prime}$, we only need to show that
  $(\operatorname{id}|_\mathcal{A},\operatorname{id}|_\mathfrak{g})$
  satisfies the properties in Definition ~\ref{definition}. This fact
  determines the metric $g$ on $\mathcal{K}$.
  
  $(1)$ $\psi(a\alpha)=a\alpha=\phi(a)\psi(\alpha)$.
  
  $(2)$ $\phi(\alpha(a))=\alpha(a)=\psi(\alpha)(\phi(a))$.
  
  $(3)$ We see that $\phi(g(\alpha,\beta))=g^{\prime}(\psi(\alpha),\psi(\beta))$.
  \begin{align*} 
    \phi(g(\alpha,\beta))&=\sum_{i,j=1}^{2}\alpha(x^i)g_{ij}\beta(x^j)\\ 
                         &= \alpha(x^1)g_{11}\beta(x^1)+\alpha(x^1)g_{12}\beta(x^2)+\alpha(x^2)g_{21}\beta(x^1)+\alpha(x^2)g_{22}\beta(x^2) \\&= \alpha(x)g_{11}\beta(x)+\alpha(x)g_{12}\beta(y)+\alpha(y)g_{21}\beta(x)+\alpha(y)g_{22}\beta(y),
  \end{align*}
  and
  \begin{align*} 
    g^{\prime}(\psi(\alpha),\psi(\beta))&=g^{\prime}(\alpha,\beta)=\sum_{i,j=1}^{3}\alpha(x^i)g_{ij}^{\prime}\beta(x^j)
    \\&=\alpha(x^1)g_{11}^{\prime}\beta(x^1)+\alpha(x^1)g_{12}^{\prime}\beta(x^2)+\alpha(x^1)g_{13}^{\prime}\beta(x^3)+\alpha(x^2)g_{21}^{\prime}\beta(x^1)\\&\quad+\alpha(x^2)g_{22}^{\prime}\beta(x^2)+\alpha(x^2)g_{23}^{\prime}\beta(x^3)+\alpha(x^3)g_{31}^{\prime}\beta(x^1)+\alpha(x^3)g_{32}^{\prime}\beta(x^2)\\&\quad+\alpha(x^3)g_{33}^{\prime}\beta(x^3)\\&= \alpha(x)g_{11}^{\prime}\beta(x)+\alpha(x)g_{12}^{\prime}\beta(y)+\alpha(x)g_{13}^{\prime}\beta(z)+\alpha(y)g_{21}^{\prime}\beta(x)+\alpha(y)g_{22}^{\prime}\beta(y)\\& \quad+\alpha(y)g_{23}^{\prime}\beta(z) 
    \alpha(z)g_{31}^{\prime}\beta(x)+\alpha(z)g_{32}^{\prime}\beta(y)+\alpha(z)g_{33}^{\prime}\beta(z)\\&= \alpha(x)g_{11}^{\prime}\beta(x)+\alpha(x)g_{12}^{\prime}\beta(y)+\alpha(y)g_{21}^{\prime}\beta(x)+\alpha(y)g_{22}^{\prime}\beta(y),  
  \end{align*} 
  since, $\alpha=a_i\{x^i,.\}=a_1\{x,.\}+a_2\{y,.\}\in\mathfrak{g}$,
  we get $\alpha(z)=a_1\{x,z\}+a_2\{y,z\}=0$.
  
  Similarly, $\beta(z)=0$. By comparing both sides
  $\phi(g(\alpha,\beta))$ and $g^{\prime}(\psi(\alpha),\psi(\beta))$,
  we conclude that $g_{11}=g_{11}^{\prime}$ ,$g_{12}=g_{12}^{\prime}$
  ,$g_{21}=g_{21}^{\prime}$ and $g_{22}=g_{22}^{\prime}$. Therefore,
  $\phi(g(\alpha,\beta))=g^{\prime}(\psi(\alpha),\psi(\beta))$.
  
  $(4)$ By construction we have that $\phi(\mathcal{A}_\text{fin})\subseteq \mathcal{A}^{\prime}_{\text{fin}} $.\\
  Therefore, $\mathcal{K}$ is a K\"ahler-Poisson subalgebra of
  $\mathcal{K}^{\prime}$ and $g$ obtained from $g^{\prime}$ as above.
  	
\end{example}

Let us take another example with different number of generators.

\begin{example}
  Let $\mathcal{K}=(\mathcal{A},g,\{x,y\})$ and
  $\mathcal{K}^{\prime}=(\mathcal{A}^{\prime},g^{\prime},\{x,y,z,w\})$
  be K\"ahler-Poisson algebras, where $\mathcal{A}$ is a subalgebra of
  $\mathcal{A}^{\prime}$ generated by $\{x,y\}$ and
  $\mathcal{A}^{\prime}$ generated by $\{x,y,z,w\}$. Let
  $\{x,y\}=p(x,y)$, $\{z,w\}=q(z,w)$ and
  $\{x,z\}=\{x,w\}=\{y,z\}=\{y,w\}=0$. To show that $\mathcal{K}$ is a
  K\"ahler-Poisson subalgebra of $\mathcal{K}^{\prime}$ we only need
  to show that
  $(\operatorname{id}|_\mathcal{A},\operatorname{id}|_\mathfrak{g})$
  satisfies the properties in Definition \ref{definition}. This fact
  determines the metric $g$ on $\mathcal{K}$.
  
  $(1)$ $\psi(a\alpha)=a\alpha=\phi(a)\psi(\alpha)$.
  
  $(2)$ $\phi(\alpha(a))=\alpha(a)=\psi(\alpha)(\phi(a))$.
  
  $(3)$ We see that $\phi(g(\alpha,\beta))=g^{\prime}(\psi(\alpha),\psi(\beta))$, the left hand side is
  \begin{align*} 
    \phi(g(\alpha,\beta))&=\sum_{i,j=1}^{2}\alpha(x^i)g_{ij}\beta(x^j)\\ 
                         &= \alpha(x^1)g_{11}\beta(x^1)+\alpha(x^1)g_{12}\beta(x^2)+\alpha(x^2)g_{21}\beta(x^1)+\alpha(x^2)g_{22}\beta(x^2) \\&= \alpha(x)g_{11}\beta(x)+\alpha(x)g_{12}\beta(y)+\alpha(y)g_{21}\beta(x)+\alpha(y)g_{22}\beta(y),
  \end{align*}
  and the right hand side is
  \begin{align*} 
    g^{\prime}(\psi(\alpha),\psi(\beta))&=g^{\prime}(\alpha,\beta)=\sum_{i,j=1}^{4}\alpha(x^i)g_{ij}^{\prime}\beta(x^j)
	\\&=\alpha(x^1)g_{11}^{\prime}\beta(x^1)+\alpha(x^1)g_{12}^{\prime}\beta(x^2)+\alpha(x^1)g_{13}^{\prime}\beta(x^3)+\alpha(x^1)g_{14}^{\prime}\beta(x^4) \\& \quad +
    \alpha(x^2)g_{21}^{\prime}\beta(x^1)+\alpha(x^2)g_{22}^{\prime}\beta(x^2)+\alpha(x^2)g_{23}^{\prime}\beta(x^3)+\alpha(x^2)g_{24}^{\prime}\beta(x^4) \\& \quad+
    \alpha(x^3)g_{31}^{\prime}\beta(x^1)+\alpha(x^3)g_{32}^{\prime}\beta(x^2)+\alpha(x^3)g_{33}^{\prime}\beta(x^3)+\alpha(x^3)g_{34}^{\prime}\beta(x^4)\\& \quad+
    \alpha(x^4)g_{41}^{\prime}\beta(x^1)+\alpha(x^4)g_{42}^{\prime}\beta(x^2)+\alpha(x^4)g_{43}^{\prime}\beta(x^3)+\alpha(x^4)g_{44}^{\prime}\beta(x^4)\\&= \alpha(x)g_{11}^{\prime}\beta(x)+\alpha(x)g_{12}^{\prime}\beta(y)+\alpha(x)g_{13}^{\prime}\beta(z)+\alpha(x)g_{14}^{\prime}\beta(w) \\& \quad +
    \alpha(y)g_{21}^{\prime}\beta(x)+\alpha(y)g_{22}^{\prime}\beta(y)+\alpha(y)g_{23}^{\prime}\beta(z)+\alpha(y)g_{24}^{\prime}\beta(w) \\& \quad+
    \alpha(z)g_{31}^{\prime}\beta(x)+\alpha(z)g_{32}^{\prime}\beta(y)+\alpha(z)g_{33}^{\prime}\beta(z)+\alpha(z)g_{34}^{\prime}\beta(w)\\& \quad+
    \alpha(w)g_{41}^{\prime}\beta(x)+\alpha(w)g_{42}^{\prime}\beta(y)+\alpha(w)g_{43}^{\prime}\beta(z)+\alpha(w)g_{44}^{\prime}\beta(w) \\&= \alpha(x)g_{11}^{\prime}\beta(x)+\alpha(x)g_{12}^{\prime}\beta(y)+\alpha(y)g_{21}^{\prime}\beta(x)+\alpha(y)g_{22}^{\prime}\beta(y),  
  \end{align*}
  since,  $\alpha=a_i\{x^i,.\}=a_1\{x,.\}+a_2\{y,.\}\in\mathfrak{g}$, we get $\alpha(z)=a_1\{x,z\}+a_2\{y,z\}=0$. 
  
  Similarly, $\alpha(w)=\beta(z)=\beta(w)=0$. By comparing both sides
  $\phi(g(\alpha,\beta))$ and $g^{\prime}(\psi(\alpha),\psi(\beta))$,
  we conclude that $g_{11}=g_{11}^{\prime}$ ,$g_{12}=g_{12}^{\prime}$
  ,$g_{21}=g_{21}^{\prime}$ and $g_{22}=g_{22}^{\prime}$. Therefore,
  $\phi(g(\alpha,\beta))=g^{\prime}(\psi(\alpha),\psi(\beta))$.
	
  $(4)$ By construction we have that $\phi(\mathcal{A}_\text{fin})\subseteq \mathcal{A}^{\prime}_{\text{fin}} $.\\
  Therefore, $\mathcal{K}$ is a K\"ahler-Poisson subalgebra of
  $\mathcal{K}^{\prime}$ and $g$ obtained from $g^{\prime}$ as above.
		
\end{example}

Above we have given examples of when
$\mathcal{K}=(\mathcal{A},g,\{x^1,...,x^m\})$ is a K\"ahler-Poisson
subalgebra of
$\mathcal{K}^{\prime}=(\mathcal{A}^{\prime},g^{\prime},\{y^1,...,{y^m}^{\prime}\})$,
where $\mathcal{A}$ is a Poisson subalgebra of
$\mathcal{A}^{\prime}$. Next proposition shows that, in general a
morphism
$(\phi,\psi):(\mathcal{K},\mathfrak{g})\rightarrow
(\mathcal{K}^{\prime},\mathfrak{g}^{\prime})$ induces a
K\"ahler-Poisson subalgebra of $\mathcal{K}^{\prime}$, denoted by
$\operatorname{Im}(\phi,\psi)$.

\begin{proposition} 
  Let $\mathcal{K}=(\mathcal{A},g,\{x^1,...,x^m\})$ and
  $\mathcal{K}^{\prime}=(\mathcal{A}^{\prime},g^{\prime},\{y^1,...,{y^m}^{\prime}\})$
  be K\"ahler-Poisson algebras, and let
  $(\phi,\psi):\mathcal{K}\rightarrow \mathcal{K}^{\prime}$ be a
  homomorphism of K\"ahler-Poisson algebras. If
  $y^J\in\phi(\mathcal{A})$ for $J=1,...,m^{\prime}$ then
  $\operatorname{Im}(\phi,\psi)=(\phi(\mathcal{A}),\tilde{g},\{\phi(x^1),\phi(x^2),...,\phi(x^m)\})$
  is a K\"ahler-Poisson subalgebra of $\mathcal{K}^{\prime}$,
  where \begin{align}
    \tilde{g}_{kl}=\phi(\eta\mathcal{P}_{km})\{\phi(x^m),y^J\}^{\prime}g^{\prime}_{JM}\phi(\eta\mathcal{P}_{ln})\{\phi(x^n),y^M\}^{\prime}.
\end{align}
 
\end{proposition}

\begin{proof}
  First, we show that $\tilde{g}$ is symmetric.
  \begin{align*} 
    \tilde{g}_{kl}&=\phi(\eta\mathcal{P}_{km})\{\phi(x^m),y^J\}^{\prime}g^{\prime}_{JM}\phi(\eta\mathcal{P}_{ln})\{\phi(x^n),y^M\}^{\prime}\\&=\phi(\eta\mathcal{P}_{ln})\{\phi(x^n),y^M\}^{\prime}g^{\prime}_{JM}\phi(\eta\mathcal{P}_{km})\{\phi(x^m),y^M\}^{\prime}\\&=\phi(\eta\mathcal{P}_{ln})\{\phi(x^n),y^M\}^{\prime}g^{\prime}_{MJ}\phi(\eta\mathcal{P}_{km})\{\phi(x^m),y^M\}^{\prime},
  \end{align*} 
  since, $g^{\prime}_{JM}=g^{\prime}_{MJ}$ we obtain
  $\tilde{g}_{kl}=\phi(\eta\mathcal{P}_{lm})\{\phi(x^m),y^J\}^{\prime}g^{\prime}_{JM}\phi(\eta\mathcal{P}_{kn})\{\phi(x^n),y^M\}^{\prime}=\tilde{g}_{lk}$.

  As $\phi(\mathcal{A})$ is a subalgebra of $\mathcal{A}^{\prime}$, to
  show that $\operatorname{Im}(\phi,\psi)$ is a subalgebra of
  $\mathcal{K}^{\prime}$, we need to show that
  $(\operatorname{id}|_{\phi(\mathcal{A})},\operatorname{id}|_{\tilde{\mathfrak{g}}})$
  is a morphism of K\"ahler-Poisson algebras, where
  $\tilde{\mathfrak{g}}$ is module of derivations of
  $\operatorname{Im}(\phi,\psi)$. Now, let
  $\alpha=\alpha_i\{\phi(x^i),.\}^{\prime}$ and
  $\beta=\beta_j\{\phi(x^j),.\}^{\prime}$ be arbitrary elements of
  $\mathfrak{\tilde{g}}$ as usual, we have

  \begin{enumerate}
  \item $\psi(a\alpha)=a\alpha=\phi(a)\psi(\alpha)$, for all $\alpha\in \tilde{\mathfrak{g}}$ and $a \in \phi(\mathcal{A})$.
  \item $\phi(\alpha(a))=\alpha(a)=\psi(\alpha)(\phi(a))$, for all $\alpha\in \tilde{\mathfrak{g}}$ and $a \in \phi(\mathcal{A})$. 
  \item $\phi(\tilde{g}(\alpha,\beta))=g^{\prime}(\psi(\alpha),\psi(\beta))=g^{\prime}(\alpha,\beta)$, for all $\alpha,\beta\in \tilde{\mathfrak{g}}$ and $a \in \phi(\mathcal{A})$.
    \begin{equation} \label{eq1}
      \begin{split}
        \phi(\tilde{g}(\alpha,\beta))-g^{\prime}(\alpha,\beta)  =& \alpha_i\{\phi(x^i),\phi(x^k)\}^{\prime}\tilde{g}_{kl} \beta_j\{\phi(x^j),\phi(x^l)\}^{\prime}\\
        & -\alpha_i\{\phi(x^i),y^J\}^{\prime}g^{\prime}_{JM} \beta_j\{\phi(x^j),y^M\}^{\prime}.
      \end{split}
    \end{equation}
    Since $y^J\in \phi(\mathcal{A})$ and for each $J$, we can find
    $\hat{y}^J\in \mathcal{A}$ such that $\phi(\hat{y}^J)=y^J$.  We
    can write
    $\tilde{g}_{kl}=\phi(\eta\mathcal{P}_{km})\{\phi(x^m),\phi(\hat{y}^J)\}^{\prime}g^{\prime}_{JM}\phi(\eta\mathcal{P}_{ln})\{\phi(x^n),\phi(\hat{y}^M)\}^{\prime}$. We
    compute
    \begin{align*} 
      \phi(\tilde{g}(\alpha,\beta))&=\phi\Big(\eta\alpha_i\{x^i,x^k\}g_{km}\{x^m,x^p\}g_{pq}\{x^q,\hat{y}^J\}\Big)g^{\prime}_{JM}
                                     \phi\Big(\eta\beta_j\{x^j,x^l\}g_{lm}\{x^m,x^r\}g_{rs}\{x^s,\hat{y}^M\}\Big)\\&=\alpha_i\{\phi(x^i),y^J\}^{\prime}g^{\prime}_{JM}\beta_j\{\phi(x^j),y^M\}^{\prime},
    \end{align*}
    by \eqref{eq1} we get
    $ \phi(\tilde{g}(\alpha,\beta))=g^{\prime}(\alpha,\beta)$.
  \item $\phi(\mathcal{A})_{\text{fin}}=\phi(\mathcal{A}_{\text{fin}})\subseteq \mathcal{A}^{\prime}_{\text{fin}} $, since, $\phi(\mathcal{A})_{\text{fin}}$ is generated by $\{\phi(x^1),...,\phi(x^m)\}$
  \end{enumerate}
  Therefore,
  $(\operatorname{id}_{\phi(\mathcal{A})},\operatorname{id}_{\tilde{\mathfrak{g}}})$
  is a morphism of K\"ahler-Poisson algebras, and
  $\operatorname{Im}(\phi,\psi)$ is a subalgebra of
  $\mathcal{K}^{\prime}$.
\end{proof}

 \section{Direct sums and tensor products of K\"ahler-Poisson algebras}
 Let $\mathcal{K}=(\mathcal{A},g,\{x^1,...,x^m\})$ and
 $\mathcal{K}^{\prime}=(\mathcal{A}^{\prime},g^{\prime},\{y^1,...,{y^m}^{\prime}\})$
 be K\"ahler-Poisson algebras. We have seen in Section 4, when
 $\mathcal{K}$ is K\"ahler-Poisson subalgebra of
 $\mathcal{K}^{\prime}$. In this section, we are interested in
 defining operations of K\"ahler-Poisson algebras. In particular, we
 introduce direct sums and tensor products of K\"ahler-Poisson
 algebras and study properties of these operations. We prove that,
 direct sums and tensor products of two K\"ahler-Poisson algebras are
 K\"ahler-Poisson algebras and we show that $\mathcal{K}$ and
 $\mathcal{K}^{\prime}$ are subalgebras of the direct sum
 $\mathcal{K}\oplus \mathcal{K}^{\prime}$.
 
 Let us first recall some basic facts of direct sums and tensor products
 of Poisson algebras. Firstly, we recall that the tensor product
 $\mathcal{A}\otimes \mathcal{A}^{\prime}$ of two Poisson algebras
 $\mathcal{A}$ and $\mathcal{A}^{\prime}$ is a Poisson algebra with
 Poisson products
 \begin{align}
  \{a_1\otimes a_2,b_1\otimes b_2\}=\{a_1,b_1\}\otimes a_2 b_2+a_1b_1\otimes \{a_2,b_2\},
 \end{align} for $a_1,b_1\in \mathcal{A}$ and $a_2,b_2\in \mathcal{A}^{\prime}$
 (see \cite{Christian}).
 
 Secondly, the direct sum of two Poisson algebras
 $\mathcal{A}\oplus \mathcal{A}^{\prime}$ is a Poisson algebra with
 Poisson product.
 \begin{align}
   \{(a_1,a_2),(b_1,b_2)\}=(\{a_1,b_1\},\{a_2,b_2\})
  \end{align} for $a_1,b_1\in \mathcal{A}$ and $a_2,b_2\in \mathcal{A}^{\prime}$.
 
  In next proposition, we see that direct sums of K\"ahler-Poisson
  algebras are K\"ahler-Poisson algebra.
  \begin{proposition}\label{proposition5.1}
    Let $\mathcal{K}=(\mathcal{A},g,\{x^1,...,x^m\})$ and
    $\mathcal{K}^{\prime}=(\mathcal{A}^{\prime},g^{\prime},\{y^1,...,{y^m}^{\prime}\})$
    be K\"ahler-Poisson algebras, and set
    $\mathcal{K}\oplus\mathcal{K}^{\prime}=(\mathcal{A}\oplus\mathcal{A}^{\prime},\hat{g},\{z^1,...,z^{m+m^{\prime}}\})$
    where
    \begin{align*}
      z^I =
      \begin{cases}
        (x^I,0)      & \quad \text{if } I\in \{1,...,m\}\\
        (0,y^{I-m})      & \quad \text{if } I\in \{m+1,...,m+m^{\prime}\},
      \end{cases}    
    \end{align*}
    and
    \begin{align*}
      \hat{g}_{IJ} =
      \begin{cases}
        (g_{IJ},0)      & \quad \text{if } I,J\in \{1,...,m\}\\
        (0,g^{\prime}_{I-m,J-m})      & \quad \text{if } I,J\in \{m+1,...,m+m^{\prime}\}\\
        (0,0) & \quad \text{otherwise.}
      \end{cases}    
    \end{align*}
    Then $\mathcal{K}\oplus\mathcal{K}^{\prime}$ is a K\"ahler-Poisson algebra.
  \end{proposition}
  
  \begin{proof}
    Since $\mathcal{K}=(\mathcal{A},g,\{x^1,...,x^m\})$ and
    $\mathcal{K}^{\prime}=(\mathcal{A}^{\prime},g^{\prime},\{y^1,...,{y^m}^{\prime}\})$
    are K\"ahler-Poisson algebras then there exists
    $\eta\in \mathcal{A}$ and $\eta^{\prime}\in \mathcal{A}^{\prime}$
    such that
    \begin{align}
      &\sum\limits_{i,j,k,l=1}^m\eta\{a_1,x^i\}g_{ij}\{x^j,x^k\}g_{kl}\{x^l,b_1\}=-\{a_1,b_1\}\\&\sum\limits_{\alpha,\beta,\gamma,\delta=1}^{m^\prime}\eta^{\prime}\{a_2,y^\alpha\}^{\prime}g^{\prime}_{\alpha\beta}\{y^\beta,y^\gamma\}^{\prime}g^{\prime}_{\gamma\delta}\{y^\delta,b_2\}^{\prime}=-\{a_2,b_2\}^{\prime}.
    \end{align}
    where $a_1,b_1\in \mathcal{A}$ and
    $a_2,b_2\in\mathcal{A}^{\prime}$.  We would like to show that
    $\mathcal{K}\oplus\mathcal{K}^{\prime}$ satisfies the
    K\"ahler-Poisson condition \eqref{eq:2.1}; that is
    \begin{align}
      \sum\limits_{I,J,K,L}^{m+m^{\prime}} (\eta,\eta^{\prime})\{(a_1,a_2),z^I\}\hat{g}_{IJ}\{z^J,z^K\}\hat{g}_{KL}\{z^L,(b_1,b_2)\} =-\{(a_1,a_2),(b_1,b_2)\}
    \end{align} 
    for $a_1,b_1\in \mathcal{A}$ and $a_2,b_2\in \mathcal{A}^{\prime}$.
    
    Starting from the left hand side
    \begin{align*}
      \sum_{I,J,K,L}^{m+m'} (\eta,&\eta')\{(a_1,a_2),z^I\}\hat{g}_{IJ}\{z^J,z^K\}\hat{g}_{KL}\{z^L,(b_1,b_2)\}\\
                                  &=\sum\limits_{J,K,L=1}^{m+m^{\prime}}\sum\limits_{i=1}^m(\eta,\eta^{\prime})\{(a_1,a_2),z^i\}\hat{g}_{iJ}\{z^J,z^K\}\hat{g}_{KL}\{z^L,(b_1,b_2)\}\\&+\sum\limits_{J,K,L=1}^{m+m^{\prime}}\sum\limits_{\alpha=1}^{m^\prime}(\eta,\eta^{\prime})\{(a_1,a_2),z^{\alpha+m}\}\hat{g}_{\alpha+m,J}\{z^J,z^K\}\hat{g}_{KL}\{z^L,(b_1,b_2)\}\\
                                  &=\sum\limits_{K,L=1}^{m+m^{\prime}}\sum\limits_{i,j=1}^m(\eta,\eta^{\prime})\{(a_1,a_2),z^i\}\hat{g}_{ij}\{z^j,z^K\}\hat{g}_{KL}\{z^L,(b_1,b_2)\}\\
                                  &+\sum\limits_{K,L=1}^{m+m^{\prime}}\sum\limits_{\alpha,\beta=1}^{m^\prime}(\eta,\eta^{\prime})\{(a_1,a_2),z^{\alpha+m}\}\hat{g}_{\alpha+m,\beta+m}\{z^{\beta+m},z^K\}\hat{g}_{KL}\{z^L,(b_1,b_2)\},
    \end{align*}
    since $\hat{g}_{iK}=0$ when $K > m$, and $\hat{g}_{m+\alpha,K}=0$
    when $K\leq m$. Moreover, since $\{z^i,z^K\}=0$ and
    $\{z^{\alpha+m},z^L\}=0$ if $1\leq i\leq m$,
    $1\leq \alpha\leq m^{\prime}$, $K> m$ and $L\leq m$, then
    \begin{align*}  \sum\limits_{I,J,K,L}^{m+m^{\prime}} (\eta,&\eta^{\prime})\{(a_1,a_2),z^I\}\hat{g}_{IJ}\{z^J,z^K\}\hat{g}_{KL}\{z^L,(b_1,b_2)\}\\&=\sum\limits_{i,j,k,l=1}^m(\eta,\eta^{\prime})\{(a_1,a_2),z^i\}\hat{g}_{ij}\{z^j,z^k\}\hat{g}_{kl}\{z^l,(b_1,b_2)\}\\&+\sum\limits_{\alpha,\beta,\gamma,\delta=1}^{m^\prime}(\eta,\eta^{\prime})\{(a_1,a_2),z^{\alpha+m}\}\hat{g}_{\alpha+m,\beta+m}\{z^{\beta+m},z^{\gamma+m}\}\hat{g}_{\gamma+m,\delta+m}\{z^{\delta+m},(b_1,b_2)\}.
    \end{align*}
    and
    \begin{align*} \sum\limits_{I,J,K,L}^{m+m^{\prime}}
      (\eta,&\eta^{\prime})\{(a_1,a_2),z^I\}\hat{g}_{IJ}\{z^J,z^K\}\hat{g}_{KL}\{z^L,(b_1,b_2)\}\\&=\sum\limits_{i,j,k,l=1}^m(\eta,\eta^{\prime})\{(a_1,a_2),(x^i,0)\}(g_{ij},0)\{(x^j,0),(x^k,0)\}(g_{kl},0)\{(x^l,0),(b_1,b_2)\}\\&+\sum\limits_{\alpha,\beta,\gamma,\delta=1}^{m^\prime}(\eta,\eta^{\prime})\{(a_1,a_2),(0,y^\alpha)\}(0,g^{\prime}_{\alpha\beta}\{(0,y^\beta),(0,y^\gamma)\}(0,g^{\prime}_{\gamma\delta})\{(0,y^\delta),(b_1,b_2)\}\\&=\sum\limits_{i,j,k,l=1}^m(\eta,\eta^{\prime})(\{a_1,x^i\},0)(g_{ij},0)(\{x^j,x^k\},0)(g_{kl},0)(\{x^l,b_1\},0)\\&+\sum\limits_{\alpha,\beta,\gamma,\delta=1}^{m^\prime}(\eta,\eta^{\prime})(0,\{a_2,y^\alpha\}^{\prime})(0,g^{\prime}_{\alpha\beta})(0,\{y^\beta,y^\gamma\}^{\prime})(0,g^{\prime}_{\gamma\delta})(0,\{y^\delta,b_2\}^{\prime})\\&=\sum\limits_{i,j,k,l=1}^m(\eta\{a_1,x^i\}g_{ij}\{x^j,x^k\}g_{kl}\{x^l,b_1\},0)+\sum\limits_{\alpha,\beta,\gamma,\delta=1}^{m^\prime}(0,\eta^{\prime}\{a_2,y^\alpha\}^{\prime}g^{\prime}_{\alpha\beta}\{y^\beta,y^\gamma\}^{\prime}g^{\prime}_{\gamma\delta}\{y^\delta,b_2\}^{\prime})\\&=
      (-\{a_1,b_1\},0)+(0,-\{a_2,b_2\}^{\prime})=-(\{a_1,b_1\},\{a_2,b_2\}^{\prime})=-\{(a_1,a_2),(b_1,b_2)\},
    \end{align*}
    since $\mathcal{K}$ and $\mathcal{K}^{\prime}$ are
    K\"ahler-Poisson algebras. Therefore,
    $\mathcal{K}\oplus\mathcal{K}^{\prime}$ is a K\"ahler-Poisson
    algebra.
  \end{proof}
  
  Next, given two K\"ahler-Poisson algebras $ \mathcal{K}$ and
  $\mathcal{K}^{\prime}$, we show that they
  are K\"ahler-Poisson subalgebras of the direct sum
  $\mathcal{K}\oplus\mathcal{K}^{\prime}$.
  \begin{remark}
    The module of inner derivations of $\mathcal{A}\oplus\mathcal{A}^{\prime}$ can be written as:
    \begin{align*}
      \mathfrak{g}\oplus\mathfrak{g}^{\prime}=\{(\alpha,\beta):\alpha\in\mathfrak{g}\text{ and }\beta\in\mathfrak{g}^{\prime}\}
    \end{align*}
    since given
    $\tilde{\alpha}\in\mathfrak{g}\oplus\mathfrak{g}^{\prime}$
    \begin{align*}
      \tilde{\alpha}(c,d)=\sum_{I=1}^{m+m^{\prime}}(a_I,b_I)\{z^I,(c,d)\}&=\sum_{I=1}^{m}(a_I\{x^I,c\},0)+\sum_{I=m+1}^{m+m^{\prime}}(0,b_I\{y^{I-m},d\})\\&=(\alpha(c),\beta(d))=(\alpha,\beta)(c,d),
    \end{align*}
    with
    \begin{align*}
      \alpha = \sum_{I=1}^ma_I\{x^I,\cdot\}\quad\text{and}\quad
      \beta = \sum_{I=1}^{m'}b_{I+m}\{y^I,\cdot\}.
    \end{align*}
  \end{remark}
  
  \begin{proposition}\label{proposition5.3}
    Let $\mathcal{K}=(\mathcal{A},g,\{x^1,...,x^m\})$ and
    $\mathcal{K}^{\prime}=(\mathcal{A}^{\prime},g^{\prime},\{y^1,...,{y^m}^{\prime}\})$
    be K\"ahler-Poisson algebras, then $\mathcal{K}$ and
    $\mathcal{K}^{\prime}$ are K\"ahler-Poisson subalgebras of
    $\mathcal{K}\oplus\mathcal{K}^{\prime}$. 
  \end{proposition}
  
  To prove this proposition for $\mathcal{K}$ and
  $\mathcal{K}^{\prime}$, we first show the following results:
  \begin{lemma}\label{lemma1}
    Let $\mathcal{K}=(\mathcal{A},g,\{x^1,...,x^m\})$ and
    $\mathcal{K}^{\prime}=(\mathcal{A}^{\prime},g^{\prime},\{y^1,...,{y^m}^{\prime}\})$
    be K\"ahler-Poisson algebras, and let
    $\tilde{\mathcal{A}}=\{(a,0):a\in\mathcal{A}\}$.  Then
    $\tilde{\mathcal{K}}=(\tilde{\mathcal{A}},\tilde{g},\{(x^1,0),...,(x^m,0)\})$
    is a K\"ahler-Poisson subalgebra of
    $\mathcal{K}\oplus\mathcal{K}^{\prime}$,
    where $\tilde{g}_{ij}=\big((g_{ij},0)\big)$.
  \end{lemma}
  
  \begin{proof}
    Firstly, $\tilde{\mathcal{A}}$ is a Poisson subalgebra of
    $\mathcal{A}\oplus\mathcal{A}^{\prime}$. Denote the derivation
    module
    $\tilde{\mathfrak{g}}=\{(\alpha,0)\in
    \mathfrak{g}\oplus\mathfrak{g}^{\prime}:\alpha \in \mathfrak{g}\}$
    as a submodule of
    $\mathfrak{g}\oplus\mathfrak{g}^{\prime}$. Secondly, we show that
    $(\text{id}|_{\tilde{A}},\text{id}|_{\tilde{\mathfrak{g}}})$ is a
    morphism of K\"ahler-Poisson algebras. Let
    $\phi=\text{id}|_{\tilde{A}}$ and
    $\psi=\text{id}|_{\tilde{\mathfrak{g}}}$, then by Definition
    \ref{definition} we get
    \begin{enumerate}
    \item $\psi((a,0)(\alpha,0))=\psi(a\alpha,0)=(a\alpha,0)$ and $\phi((a,0))\psi((\alpha,0))=(a,0)(\alpha,0)=(a\alpha,0)$.
      
    \item $\phi(\alpha(a,0))=\phi(a\alpha,0)=(a\alpha,0)$ and $\psi(\alpha,0)(\phi(a,0))=(\alpha,0)(a,0)=(a\alpha,0)$.
      
    \item We see that $\phi(\tilde{g}((\alpha,0),(\beta,0)))=\hat{g}(\psi(\alpha,0),\psi(\beta,0))$.
      \begin{align*}
        \phi(\tilde{g}((\alpha,0),(\beta,0)))&=\sum_{i,j=1}^{m}((\alpha,0)(x^i,0)(g_{ij},0)(\beta,0)(x^j,0))\\&=\sum_{i,j=1}^{m}(\alpha(x^i)g_{ij}\beta(x^j),0),
      \end{align*}
      and 
      \begin{align*}
        \hat{g}(\psi(\alpha,0),\psi(\beta,0))=\hat{g}((\alpha,0),(\beta,0))=(\alpha,0)(\tilde{x}^I,0)\hat{g}_{IJ}(\beta,0)(\tilde{x}^J,0).
      \end{align*}
      Since $\hat{g}_{IJ}$ is a diagonal block matrix, we get 
      \begin{align*}
	\hat{g}(\psi(\alpha,0),\psi(\beta,0))&=\sum_{I,J=1}^{m}(\alpha,0)(x^I,0)({g}_{IJ},0)(\beta,0)(x^J,0)\\&+\sum_{I,J=m+1}^{m+m^{\prime}}(\alpha,0)(0,y^I)(0,{g^{\prime}}_{I-m,J-m})(\beta,0)(0,y^J)
      \end{align*}
      since $(\alpha,0)(0,y^J)=0$, we get
      \begin{align*}
        \hat{g}(\psi(\alpha,0),\psi(\beta,0))=\sum_{i,j=1}^{m}(\alpha(x^i),0)(g_{ij},0)(\beta(x^j),0)=\sum_{i,j=1}^{m}(\alpha(x^i)g_{ij}\beta(x^j),0).
      \end{align*}
      
    \item By construction, we have $\phi(\tilde{\mathcal{A}}_\text{fin})\subseteq (\mathcal{A}\oplus\mathcal{A}^{\prime})_{\text{fin}} $.
    \end{enumerate}
    Therefore, $\tilde{\mathcal{K}}$ is a K\"ahler-Poisson subalgebra
    of $\mathcal{K}\oplus\mathcal{K}^{\prime}$.
  \end{proof}

  \begin{lemma}\label{lemma2}
    Let $\mathcal{K}=(\mathcal{A},g,\{x^1,...,x^m\})$ and
    $\tilde{\mathcal{K}}=(\tilde{\mathcal{A}},\tilde{g},\{(x^1,0),...,(x^m,0)\})$
    be K\"ahler-Poisson algebras as in Lemma~\ref{lemma1}, then $\mathcal{K}$ is
    isomorphic to $\tilde{\mathcal{K}}$.
  \end{lemma}
  
  \begin{proof}
    We see that there is a morphism
    $\phi:\mathcal{A}\rightarrow \tilde{\mathcal{A}}$ and
    $\psi:\mathfrak{g}\rightarrow \tilde{\mathfrak{g}}$ defined by
    $\phi(c)=(c,0)$ ($\phi$ is a Poisson algebra isomorphism, see
    Remark \ref{remark3.6}), for all $c\in \mathcal{A}$ and let
    $\alpha=a_i\{x^i,.\}$ we define
    $\psi(\alpha)=(a_i,0)\{(x^i,0),.\}$ and
    $\psi(\beta)=(b_i,0)\{(x^i,0),.\}$ for all
    $\alpha,\beta\in \mathfrak{g}$, we get
    \begin{enumerate}
    \item $\phi(c)\psi(\alpha)=(c,0)(a_i,0)\{(x^i,0),.\}=(ca_i,0)\{(x^i,0),.\}=\psi(c\alpha)$.
      
    \item $\psi(\alpha)(\phi(a))=(a_i,0)\{(x^i,0),(c,0)\}=(a_i,0)(\{x^i,c\},0)=(a_i\{x^i,c\},0)=\phi(a_i\{x^i,c\})=\phi(\alpha(c))$.
      
    \item We see that $\phi(g(\alpha,\beta))=\tilde{g}(\psi(\alpha),\psi(\beta))$ since
      \begin{align*}
        \phi(g(\alpha,\beta))=\sum_{i,j=1}^{m}(\alpha(x^i)g_{ij}\beta(x^j),0),
      \end{align*}
      and 
      \begin{align*}
        \tilde{g}(\psi(\alpha),\psi(\beta))&=\tilde{g}((a_l,0)\{(x^l,0),.\},(b_k,0)\{(x^k,0),.\})\\&=(a_l,0)\{(x^l,0),(x^i,0)\}\tilde{g}_{ij}(b_k,0)\{(x^k,0),(x^j,0)\}\\&=(a_l,0)(\{x^l,x^i\},0)(g_{ij},0)(b_k,0)(\{x^k,x^j\},0)\\&=(a_l\{x^l,x^i\},0)(g_{ij},0)(b_k\{x^k,x^j\},0)\\&=(\alpha(x^i),0)(g_{ij},0)(\beta(x^j),0)=(\alpha(x^i)g_{ij}\beta(x^j),0).
      \end{align*}
      
    \item By construction, we have $\phi(\mathcal{A}_\text{fin})=\tilde{\mathcal{A}}_{\text{fin}} $.
    \end{enumerate}
    Therefore, $\mathcal{K}$ isomorphic to $\tilde{\mathcal{K}}$. 
  \end{proof}

  \begin{proof}[Proof of Proposition \ref{proposition5.3}]
    To prove Proposition \ref{proposition5.3} we need the above
    results. Lemma \ref{lemma1} says that
    $\tilde{\mathcal{K}}=(\tilde{\mathcal{A}},\tilde{g},\{(x^1,0),...,(x^m,0)\})$
    is a K\"ahler-Poisson subalgebra of
    $\mathcal{K}\oplus\mathcal{K}^{\prime}$ and Lemma \ref{lemma2}
    says that $\mathcal{K}$ is isomorphic to
    $\tilde{\mathcal{K}}$. Therefore, we can consider $\mathcal{K}$ to
    be a K\"ahler-Poisson subalgebra of
    $\mathcal{K}\oplus\mathcal{K}^{\prime}$.
  \end{proof}

  \noindent
  Observe that, in the same way, $\mathcal{K}^{\prime}$ is a
  K\"ahler-Poisson subalgebra of
  $\mathcal{K}\oplus\mathcal{K}^{\prime}$.  The next proposition shows
  that the tensor product of K\"ahler-Poisson algebras is
  K\"ahler-Poisson algebra under certain conditions.
  
 \begin{proposition}
   Let $\mathcal{K}=(\mathcal{A},g,\{x^1,...,x^m\})$ and
   $\mathcal{K}^{\prime}=(\mathcal{A}^{\prime},g^{\prime},\{y^1,...,{y^m}^{\prime}\})$
   with
   \begin{center}
     $\sum\limits_{i,j,k,l=1}^m\eta\{a_1,x^i\}g_{ij}\{x^j,x^k\}g_{kl}\{x^l,b_1\}=-\{a_1,b_1\}$
  \end{center} and 
  \begin{center}
    $\sum\limits_{\alpha,\beta,\gamma,\delta=1}^{m^\prime}\eta^{\prime}\{a_2,y^\alpha\}^{\prime}g^{\prime}_{\alpha\beta}\{y^\beta,y^\gamma\}^{\prime}g^{\prime}_{\gamma\delta}\{y^\delta,b_2\}^{\prime}=-\{a_2,b_2\}^{\prime}$
  \end{center}
  be K\"ahler-Poisson algebras, and assume that there exist
  $\rho \in \mathcal{A}$ and $\rho^{\prime} \in \mathcal{A}^{\prime}$
  such that $\rho^2=\eta$ and $\rho^{\prime2}=\eta^{\prime}$. Set
  $\mathcal{K}\otimes\mathcal{K}^{\prime}=(\mathcal{A}\otimes\mathcal{A}^{\prime},\tilde{g},\{z^1,...,z^{m+m^{\prime}}\})$
  where
  \begin{align*}
    z^I =
    \begin{cases}
      x^I\otimes 1      & \quad \text{if } I\in \{1,...,m\}\\
      1\otimes y^{I-m}      & \quad \text{if } I\in \{m+1,...,m+m^{\prime}\},
    \end{cases}
  \end{align*}
  and
  \begin{align*}
   \tilde{g}_{IJ} =
    \begin{cases}
      \rho g_{IJ}\otimes 1      & \quad \text{if } I,J\in \{1,...,m\}\\
      1\otimes \rho^{\prime}g^{\prime}_{I-m,J-m}     & \quad \text{if } I,J\in \{m+1,...,m+m^{\prime}\}\\
      0 & \quad \text{otherwise}.
    \end{cases} 
  \end{align*}
  Then $\mathcal{K}\otimes\mathcal{K}^{\prime}$ is a K\"ahler-Poisson algebra.
  \end{proposition}

  \begin{proof}
    Let us show that $\mathcal{K}\otimes\mathcal{K}^{\prime}$ satisfies the K\"ahler-Poisson condition: \begin{align}
    \sum\limits_{I,J,K,L}^{m+m^{\prime}} \{a_1\otimes a_2,z^I\}\tilde{g}_{IJ}\{z^J,z^K\}\tilde{g}_{KL}\{z^L,b_1\otimes b_2\} =-\{a_1\otimes a_2,b_1\otimes b_2\}\end{align} 
    for $a_1,b_1\in \mathcal{A}$ and $a_2,b_2\in \mathcal{A}^{\prime}$.
    
    Starting from the left hand side
    \begin{align*}
      &\sum_{I,J,K,L}^{m+m'}
        \{a_1\otimes a_2,z^I\}\tilde{g}_{IJ}\{z^J,z^K\}\tilde{g}_{KL}\{z^L,b_1\otimes b_2\}\\
      &=\sum\limits_{J,K,L=1}^{m+m^{\prime}}\sum\limits_{i=1}^m\{a_1\otimes a_2,z^i\}\tilde{g}_{iJ}\{z^J,z^K\}\tilde{g}_{KL}\{z^L,b_1\otimes b_2\}\\
      &+\sum_{J,K,L=1}^{m+m'}\sum_{\alpha=1}^{m'}\{a_1\otimes a_2,z^{\alpha+m}\}\tilde{g}_{\alpha+m,J}\{z^J,z^K\}\tilde{g}_{KL}\{z^L,b_1\otimes b_2\}\\
      &=\sum_{K,L=1}^{m+m'}\sum_{i,j=1}^m\{a_1\otimes a_2,z^i\}\tilde{g}_{ij}\{z^j,z^K\}\tilde{g}_{KL}\{z^L,b_1\otimes b_2\}\\
      &+\sum\limits_{K,L=1}^{m+m^{\prime}}\sum\limits_{\alpha,\beta=1}^{m^\prime}\{a_1\otimes a_2,z^{\alpha+m}\}\tilde{g}_{\alpha+m,\beta+m}\{z^{\beta+m},z^K\}\tilde{g}_{KL}\{z^L,b_1\otimes b_2\}
    \end{align*}  
    since $\tilde{g}_{iK}=0$ when $K > m$, and
    $\tilde{g}_{m+\alpha,K}=0$ when $K\leq m$. Moreover, since
    $\{z^i,z^K\}=\{x^i\otimes 1,1\otimes y^{K-m}\}=0$ and
    $\{z^{\alpha+m},z^L\}=\{1\otimes y^{\alpha},x^L\otimes 1\}=0$ if
    $1\leq i\leq m$, $1\leq \alpha\leq m^{\prime}$, $K> m$ and
    $L\leq m$, then
    \begin{align*}
      \sum\limits_{I,J,K,L}^{m+m^{\prime}}&\{a_1\otimes a_2,z^I\}\tilde{g}_{IJ}\{z^J,z^K\}\tilde{g}_{KL}\{z^L,b_1\otimes b_2\}\\&=\sum\limits_{i,j,k,l=1}^m\{a_1\otimes a_2,z^i\}\tilde{g}_{ij}\{z^j,z^k\}\tilde{g}_{kl}\{z^l,b_1\otimes b_2\}\\&+\sum\limits_{\alpha,\beta,\gamma,\delta=1}^{m^\prime}\{a_1\otimes a_2,z^{\alpha+m}\}\tilde{g}_{\alpha+m,\beta+m}\{z^{\beta+m},z^{\gamma+m}\}\tilde{g}_{\gamma+m,\delta+m}\{z^{\delta+m},b_1\otimes b_2\},
    \end{align*}
    and
    \begin{align*}      \sum\limits_{I,J,K,L}^{m+m^{\prime}} &\{a_1\otimes a_2,z^I\}\tilde{g}_{IJ}\{z^J,z^K\}\tilde{g}_{KL}\{z^L,b_1\otimes b_2\}\\&=\sum\limits_{i,j,k,l=1}^m\{a_1\otimes a_2,x^i\otimes 1\}\rho g_{ij}\otimes 1\{x^j\otimes 1,x^k\otimes 1\}\rho g_{kl}\otimes 1\{x^l\otimes 1,b_1\otimes b_2\}\\&+\sum\limits_{\alpha,\beta,\gamma,\delta=1}^{m^\prime}\{a_1\otimes a_2,1\otimes y^\alpha\}1\otimes \rho^{\prime}g^{\prime}_{\alpha\beta}\{1\otimes y^\beta,1\otimes y^\gamma\}1\otimes \rho^{\prime}g^{\prime}_{\gamma\delta}\{1\otimes y^\delta,b_1\otimes b_2\}\\&=\sum\limits_{i,j,k,l=1}^m(\{a_1,x^i\}\otimes a_2)(\rho g_{ij}\otimes 1)(\{x^j,x^k\}\otimes 1)(\rho g_{kl}\otimes 1)(\{x^l,b_1\}\otimes b_2)\\&+\sum\limits_{\alpha,\beta,\gamma,\delta=1}^{m^\prime}(a_1\otimes \{a_2,y^\alpha\}^{\prime})(1\otimes \rho^{\prime}g^{\prime}_{\alpha\beta})(1\otimes\{ y^\beta,y^\gamma\}^{\prime})(1\otimes \rho^{\prime}g^{\prime}_{\gamma\delta})(b_1\otimes\{y^\delta, b_2\}^{\prime})\\&=\sum\limits_{i,j,k,l=1}^m\rho^2\{a_1,x^i\}g_{ij}\{x^j,x^k\}g_{kl}\{x^l,b_1\}\otimes a_2 b_2\\&+\sum\limits_{\alpha,\beta,\gamma,\delta=1}^{m^\prime} a_1b_1 \otimes \rho^{\prime2}\{a_2,y^\alpha\}^{\prime}g^{\prime}_{\alpha\beta}\{ y^\beta,y^\gamma\}^{\prime}g^{\prime}_{\gamma\delta}\{y^\delta, b_2\}^{\prime}\\&=\sum\limits_{i,j,k,l=1}^m(\eta\{a_1,x^i\}g_{ij}\{x^j,x^k\}g_{kl}\{x^l,b_1\}\otimes a_2 b_2)\\&+\sum\limits_{\alpha,\beta,\gamma,\delta=1}^{m^\prime}( a_1b_1 \otimes \eta^{\prime}\{a_2,y^\alpha\}^{\prime}g^{\prime}_{\alpha\beta}\{ y^\beta,y^\gamma\}^{\prime}g^{\prime}_{\gamma\delta}\{y^\delta, b_2\}^{\prime}),
    \end{align*}
    since $\rho^2=\eta$ and $\rho^{\prime2}=\eta^{\prime}$.  Finally,
    since $\mathcal{K}$ and $\mathcal{K}^{\prime}$ are
    K\"ahler-Poisson algebras we get
    \begin{align*}\sum\limits_{I,J,K,L}^{m+m^{\prime}}\{a_1\otimes a_2,z^I\}\tilde{g}_{IJ}\{z^J,z^K\}\tilde{g}_{KL}\{z^L,b_1\otimes b_2\}&=-\{a_1,b_1\}\otimes a_2b_2-a_1b_1\otimes\{a_2,b_2\}^{\prime}\\&=-\{a_1\otimes a_2,b_1\otimes b_2\}.
    \end{align*}    
    Therefore, $\mathcal{K}\otimes \mathcal{K}^{\prime}$ is a K\"ahler-Poisson algebra.
   \end{proof}

   \section{Summary}
   In this paper, we have introduced the concept of morphism of
   K\"ahler-Poisson algebras. We have recalled a few results from
   \cite{algebras}, in order to motivate and understand the concept of
   morphism of K\"ahler-Poisson algebras. We have studied properties
   of isomorphisms for K\"ahler-Poisson algebras and we illustrates
   with examples when two K\"ahler-Poisson algebras are isomorphic. We
   have used the concept of morphism to define subalgebras of
   K\"ahler-Poisson algebras and we have presented examples when
   $\mathcal{K}$ is a K\"ahler-Poisson subalgebra of
   $\mathcal{K}^{\prime}$, where $\mathcal{A}$ is a proper Poisson
   subalgebra of a finitely generated algebra
   $\mathcal{A}^{\prime}$. Finally, in Section 5, we have introduced
   direct sums and tensor products of K\"ahler-Poisson algebras and
   properties of these operations.
   
   There are many open questions that one would like to investigate in
   future work. For instance, is there a natural way to study the
   moduli spaces of Poisson algebras; i.e. how many (non-isomorphic)
   K\"ahler-Poisson structures does there exist on a given Poisson
   algebra?

   \section*{Acknowledgements}
   I would like to thank my supervisor J. Arnlind for fruitful
   discussions and helpful comments. I would also like to thank my
   co-supervisor M. Izquierdo for ideas and
   discussions. 
   The results in Section 3 are a part of my licentiate thesis
   \cite{thesis}.

\end{document}